\documentclass[11pt,a4paper]{article}
\usepackage[T1]{fontenc}
\usepackage[utf8]{inputenc}
\usepackage[english]{babel}
\usepackage{amsfonts}
\usepackage{mathrsfs}

\usepackage{lmodern,microtype,amsmath,amssymb,amsthm,mathtools}
\usepackage[
  top=30mm,
  bottom=30mm,
  left=35mm,
  right=35mm
]{geometry}
\usepackage{enumitem,booktabs}
\usepackage[colorlinks=true,linkcolor=blue,citecolor=blue,urlcolor=blue,linktoc=page,
 pdfauthor={Timothee Benard},
 pdftitle={Proof of the singularity conjecture for discrete subgroups of SL(2,R)}]{hyperref}
\usepackage[capitalise,noabbrev]{cleveref}
\numberwithin{equation}{section}
\newtheorem{theorem}{Theorem}[section]
\newtheorem{proposition}[theorem]{Proposition}
\newtheorem{lemma}[theorem]{Lemma}
\newtheorem{corollary}[theorem]{Corollary}
\newtheorem*{conjecture}{Conjecture}
\theoremstyle{definition}

\newtheorem{remark}[theorem]{Remark}
\newcommand{\R}{\mathbb R}
\newcommand{\C}{\mathbb C}
\newcommand{\Z}{\mathbb Z}
\newcommand{\N}{\mathbb N}

\newcommand{\D}{\mathbb D}
\newcommand{\E}{\mathbb E}
\newcommand{\Prob}{\mathbb P}
\newcommand{\Id}{\mathrm{Id}}
\newcommand{\dd}{\,\mathrm d}
\newcommand{\rot}{\mathsf r}

\newcommand{\Utor}{U_{\mathrm{tor}}}
\newcommand{\norm}[1]{\lVert#1\rVert}
\DeclareMathOperator{\SL}{SL}
\DeclareMathOperator{\PSL}{PSL}
\DeclareMathOperator{\SO}{SO}

\DeclareMathOperator{\supp}{supp}
\DeclareMathOperator{\Area}{Area}

\DeclareMathOperator{\sech}{sech}

\DeclareMathOperator{\Acc}{Acc}
\DeclareMathOperator{\Bil}{Bil}
\setlist{topsep=4pt,itemsep=2pt}
\DeclareMathOperator{\leb}{leb}
\title{Proof of the singularity conjecture for discrete subgroups of $\SL_2(\R)$}
\author{Timoth\'ee B\'enard}
\date{}

\begin{document}
\maketitle
\begin{abstract}
Let $\mu$ be a finitely supported probability measure on $\SL_2(\R)$
whose support generates a  Zariski-dense discrete subgroup.
We prove that its stationary probability measure on  $\mathbb P^1(\R)$
is singular with respect to the $\SO(2)$-invariant probability measure.
The proof compares winding variances on finite covers of a closed
hyperbolic surface and uses a topological obstruction.
\end{abstract}

\setcounter{tocdepth}{2}
\tableofcontents

\section{Introduction}
\label{sec:intro}
Let $\D$ be the Poincaré disk, endowed with the usual action of $\PSL_{2}(\R)$.
Let $\Gamma\leq \PSL_{2}(\R)$ be a  Zariski-dense subgroup. Equivalently, $\Gamma$ is non-elementary, i.e. it has no finite orbit on $\D\cup \partial \D$, or  equivalently  again,  it contains two hyperbolic elements with disjoint sets of fixed points on $\partial \D$. Let $\mu$ be a probability measure on $\Gamma$ which is generating, i.e. whose support generates $\Gamma$ as a group. It has been known since Furstenberg \cite{Furstenberg63} that $\partial \D$ carries a unique probability measure $\nu$ satisfying the stationarity relation 
\[
                 \nu=\mu*\nu:=\int g_*\nu\dd\mu(g).
\]
We call $\nu$ the Furstenberg measure associated to $\mu$. It is also called the hitting measure on $\partial \D$, due to the fact that for $\mu^{\otimes \N^*}$-almost every  $\underline{\gamma}=(\gamma_{i})_{i\geq 1}$, the orbit trajectory $\gamma_{1}\dots \gamma_{n}o$ ($o\in \D$ arbitrary) converges to some limit point $\xi_{\underline{\gamma}}\in \partial \D$, whose distribution under $\mu^{\otimes \N^*}$ is $\nu$. This paper is concerned with the regularity of $\nu$ when $\Gamma$ is a \emph{discrete} subgroup. More precisely,  denoting by $\leb$ the angular probability measure on $\partial \D$, we study whether $\nu$ is singular or absolutely continuous with respect to $\leb$. 

In his 1971 paper, Furstenberg \cite{Furstenberg} provided a first positive result regarding the potential regularity of $\nu$. He showed that every lattice \(\Gamma<\mathrm{PSL}_2(\mathbb R)\) carries a generating probability measure with finite first hyperbolic moment whose hitting measure $\nu$ is exactly the angular measure $\leb$.
Surprisingly,  in the early 1990s, Guivarc'h and Le Jan obtained a strong result in the opposite direction: for $\Gamma=\Gamma(2)$ the level-two congruence subgroup of $\PSL_{2}(\Z)$, and any generating measure $\mu$ with finite first moment for a \emph{word} metric, the hitting measure is  singular: $\nu \perp \leb$ \cite{GLJ90,GLJ93,GLJ96}. Their proof is based on comparing the homological winding for typical geodesic rays chosen by  $\nu$ and $\leb$ wrapping around  the thrice-punctured sphere $\Gamma(2)\backslash \D$, and exhibiting different limit behaviors.

Note that compared to Furstenberg's setting, Guivarc'h and Le Jan imposed a stronger moment assumption on $\mu$. Pushing this intuition, Kaimanovich and Le Prince formulated the following general singularity conjecture  \cite{KLP}. 

\begin{conjecture}[Kaimanovich--Le Prince] 
Every probability measure $\mu$ on $\PSL_{2}(\R)$ (or more generally $\PSL_{d}(\R)$) whose support is finite and generates a  Zariski-dense  discrete subgroup $\Gamma$ must have singular hitting measure  at infinity. 
\end{conjecture}

In fact, the original version of the conjecture did not assume  discreteness. It was  motivated by Kaimanovich--Le Prince's result \cite{KLP}: every finitely generated Zariski-dense subgroup of $\PSL_{d}(\R)$, discrete or not, carries a finitely supported symmetric generating measure which is singular at infinity. However, as shown later by B\'ar\'any--Pollicott--Simon  \cite{BaranyPollicottSimon} and Bourgain \cite{Bourgain},  $\PSL_{2}(\R)$ also contains finitely supported   measures $\mu$ with absolutely continuous hitting measure, see also Benoist--Quint  \cite{BQ} and Lequen \cite{Lequen} for extensions to more general semisimple  groups. The measures involved in all these counterexamples span a dense subgroup (for the Lie group topology).  Imposing that the group generated by $\supp \mu$ be discrete has been adopted by the community as the natural restriction to make the conjecture potentially valid. 

Since its formulation, the singularity  conjecture was  established in various contexts.
 For $\Gamma<\PSL_{2}(\R)$ a \emph{non-uniform lattice},  the cusp excursion approach initiated by Guivarc'h--Le Jan was developed further by Gadre--Maher--Tiozzo \cite{GMT}  and the author \cite{Benard}, verifying singularity under a finite first word moment. Alternative approaches were also put forward by Deroin--Kleptsyn--Navas \cite{DeroinKleptsynNavas}, comparing Lyapunov expansion exponents for $\nu$-typical and $\leb$-typical points on the circle, and by Blach\`ere--Ha\"issinsky--Mathieu \cite{BlachereHaissinskyMathieu} showing that for  $\mu$ finitely supported symmetric,  $\nu \not\perp \leb$ would force $\Gamma$ to be quasi-isometric to $\D$, which is impossible because their Gromov boundaries are different. 
Beyond $\PSL_{2}(\R)$, Randecker–Tiozzo \cite{RT} extended the cusp-excursion approach to higher-dimensional real hyperbolic settings.

For a Zariski-dense discrete group $\Gamma<\PSL_{2}(\R)$ of \emph{infinite covolume}, Kim–Zimmer \cite[Theorem 1.5]{KZ} established singularity under a finite first hyperbolic-moment assumption, using geodesic tracking and occupation-time estimates. Kim and Zimmer also establish higher-dimensional and higher-rank results. The higher-rank case was also studied independently by  Lee--Tiozzo--Van LimBeek \cite{LVT} for symmetric walks. 
Let us point out that for finitely supported measures, the conjecture can be checked directly: as $\Gamma$ is finitely generated of infinite covolume, its limit set $\overline{\Gamma o}\cap \partial \D$ has angular measure zero \cite{Beardon0}, while $\nu$ is supported on it.

The remaining case  of $\Gamma$ a \emph{uniform lattice} has remained mysterious. Partial results were obtained in the past few years. Carrasco--Lessa--Paquette \cite{CLP} and Kosenko \cite{Kosenko} proved singularity for certain random walks on cocompact Fuchsian groups associated with \emph{regular} hyperbolic polygons, Kosenko's result covering all but finitely many regular-polygon cases. Kosenko--Tiozzo \cite{KT} subsequently extended Kosenko's result to \emph{centrally symmetric} polygons and removed certain symmetry assumptions; see also \cite{KosenkoFirstPassage} for walks supported on powers of the side-pairing generators. The common principle of these proofs is to establish the strict entropy--drift inequality \(h<\ell\), and then use the entropy--drift dimension formula to conclude that \(\dim \nu=h/\ell<1\). A complementary complex-analytic approach, based on Cauchy
transforms and generalized analytic continuation, was developed
by Kosenko \cite{KosenkoComplex}.
Closer to the spirit of Guivarc'h--Le Jan, the author \cite{Benard}  used winding techniques to show that absolute
continuity at infinity constrains  $\mu$ to be centered in the abelianization, with a specific  covariance, thus ruling out many candidate measures. Recently,  Bogachev--Kosenko--Tiozzo \cite{BKT}  validated the conjecture for sufficiently small cocompact deformations of non-uniform lattices equipped with a finitely supported generating measure. They also obtained related results for Kleinian groups.
\bigskip

We prove the unrestricted  statement for discrete subgroups of $\PSL_{2}(\R)$.

\begin{theorem}
\label{thm:main}
Let $\mu$ be a probability measure on $\PSL_2(\R)$ whose support is finite and generates a Zariski-dense discrete  
subgroup of $\PSL_{2}(\R)$. Then its stationary probability measure $\nu$ on
$\partial \D$ is singular with respect to $\leb$.
\end{theorem}

As explained above, we only need to deal with the case where $\Gamma$ is a discrete cocompact subgroup of $\PSL_{2}(\R)$.

\medskip


The \emph{mechanism of the proof} goes as follows. Let $S_{\Gamma}:=\Gamma\backslash \D$ be a closed hyperbolic surface, of genus $g\geq 2$.
By the author's earlier work \cite{Benard},  a typical geodesic ray on $S_{\Gamma}$ selected by the hitting measure $\nu$ must satisfy a certain list of  homological winding statistics (law of large numbers, central limit theorem, etc.) when the time parameter goes to infinity. Those statistics involve limiting objects, such as mean and covariance. The main idea is to examine those statistics simultaneously for all finite cyclic covers of $S_{\Gamma}$. 

To do so, we use torsion points of the torus $(S^1)^{2g}$ to encode
the finite cyclic covers of $S_{\Gamma}$. Above each nontrivial
torsion point, we attach a certain vector space of one-forms on
the cover. Twisted cohomology allows to organize and extend this
structure into a vector bundle over
$(S^1)^{2g}\smallsetminus\{1\}$. The limiting winding covariance
discussed above defines on the fibre of every nontrivial torsion
point a certain Hermitian form, which is explicit in terms of
the driving measure $\mu$ (although it only depends on $\nu$).
If $\mu$ has finite support, this field of Hermitian forms is
seen to extend rationally over the punctured torus $(S^1)^{2g}\smallsetminus\{1\}$. On the other
hand, if $\nu$ is not singular with $\leb$, the winding covariance
also constrains the so-called Hodge-star automorphism of the
bundle, as can be deduced from a theorem of Ratner \cite{Ratner}
about winding for the Lebesgue case. As a consequence, the
Hodge-star automorphism itself needs to be rational in the torus
coordinates, and this forces the determinant of the
$(1,0)$-subbundle to be trivial. Its identification with a Picard
bundle over the Jacobian shows on the other hand that its first
Chern class is nonzero. Contradiction.


 \begin{remark}
A related use of the character torus to organize homology across
finite abelian covers appears in Koberda~\cite{Koberda}, who studies
the spectra of lifts of pseudo-Anosov homeomorphisms on the homology
of suitable finite abelian covers. In our setting, this viewpoint
allows us to compare winding covariance forms across covers and
derive a topological obstruction.
 \end{remark}


We now give a more detailed summary. It follows the chronology of the paper, and we encourage the reader to use it as a road map later on.

\medskip
\noindent{\bf Detailed summary of the proof}.
Assume that $\Gamma$ is cocompact, discrete, and torsion-free.  Set $S_{\Gamma}=\Gamma\backslash\D$ the associated closed hyperbolic surface, let $g\geq 2$ be the genus of $S_{\Gamma}$. 

Characters of $\Gamma$ are parametrized by $ (\C^*)^{2g}$. More precisely,   every $q\in (\C^*)^{2g}$ defines  a character $\chi_{q}: \Gamma \rightarrow \C^*$ by prescribing its values on a fixed basis of $\Gamma/[\Gamma, \Gamma]\simeq \Z^{2g}$. We write $H^1_{\chi_{q}}$ the cohomology space of $\chi_{q}$-equivariant closed complex-valued one-forms over $\D$, see \S \ref{subsec:tc-cohomology}. For $q \neq1$, we have $\dim H^1_{\chi_{q}}=2g-2$, but for $q=1$ the dimension jumps to $2g$ so we will always remove this point. We set $X=(\C^*)^{2g} \smallsetminus \{1\}$. The union $\mathcal H^1_{X}:=\sqcup_{q\in X} H^1_{\chi_{q}}$ has a natural structure of algebraic vector bundle over $X$, that we call the twisted cohomology bundle (\S \ref{subsec:tc-bundle}). Its restriction $\mathcal H^1_{U}$ to the unitary locus $U=(S^1)^{2g} \smallsetminus \{1\}$  carries a  standard smooth Hermitian structure,   and a smooth automorphism of order $4$, given by the Hodge-star (\S \ref{subsec:tc-unitary}). Finally, the subset  $\Utor$ of torsion points in $U$ has a special significance. Indeed, it encodes characters with finite image, therefore the fibre $ H^1_{\chi_{q}}$ above every $q\in \Utor$ can be seen as made of one-forms on a finite cyclic cover of $S_{\Gamma}$, given by $S_{\Lambda_{q}}:=\ker \chi_{q}\backslash\D$, see \S \ref{subsec:tc-torsion}.

Let $\mu$ be a  generating measure on $\Gamma$ with finite exponential moment. Set $\nu$ its Furstenberg measure on $\partial \D$. For $\xi \in \partial \D$, write  $r_{\xi}: \R^+\rightarrow \D$ the unit-speed geodesic ray from the origin $o\in \D$ to $\xi$. Given a torsion point $q\in \Utor$, a $\chi_{q}$-equivariant closed one-form $\omega$ on $\D$, we may consider a  direction $\xi$ chosen randomly by $\nu$, and study the statistics of the $\omega$-winding $f_{\omega}(r_\xi(t))=\int_{[o, r_{\xi}(t)]} \omega$. Various limit theorems for this have been established by the author in the former work \cite{Benard}, see  \S \ref{sec:winding}, \S \ref{sec-windcovform}. In particular, a central limit theorem yields a limiting variance, which can be represented by a bilinear form $B_{\nu, \chi_{q}}$ on $H^1_{\chi_{q}}\times H^1_{\chi_{q^{-1}}}$ depending explicitly on $\mu$. On the other hand, in the case $\nu =  \leb$, a central limit theorem of Ratner \cite{Ratner} allows us to derive an explicit geometric formula for $B_{\leb, \chi_{q}}$, \S \ref{subsec:geometric}. As soon as $\nu \not \perp \leb$, those need to coincide, i.e.  $B_{\nu, \chi_{q}}=B_{\leb, \chi_{q}}$, in view of the law of the iterated logarithm from \cite{Benard}. 

Assume now $\mu$ finitely supported and $\nu \not \perp \leb$. The finite support assumption implies that  $q\mapsto B_{\nu, \chi_{q}}$ depends rationally on $q$, and therefore extends beyond the unitary locus $U$, to some  Zariski-open neighborhood $O \subseteq X$ (\S \ref{subsec:finisuprat}). 
Using $B_{\nu, \chi_{q}}= B_{\leb, \chi_{q}}$ (due to $\nu \not \perp \leb$) and invoking the geometric formula from Ratner for $B_{\leb, \chi_{q}}$, 
we derive that the Hodge-star automorphism, initially defined on $ \mathcal H^1_{U}$, can be extended to $\mathcal H^1_{O}$ as a \emph{regular} automorphism (\S \ref{subsec-rathodsbet}). This allows to define an associated  algebraic subbundle $\mathcal H^{(1,0)}_{O}$ of rank $g-1$ over $O$. Now, on the one hand, the determinant bundle $\bigwedge^{g-1} \mathcal H^{(1,0)}_{O}$ must be trivial, because any algebraic line bundle over a  Zariski-open subset of $(\C^{*})^{2g}$ is known to be trivial. On the other hand,  interpreting $\mathcal H^{(1,0)}_{U}$ as the canonical Picard
bundle over the punctured Jacobian of $S_{\Gamma}$, holomorphic bundle theory allows us to compute its first Chern class and directly see it is nontrivial, so $\bigwedge^{g-1} \mathcal H^{(1,0)}_{U}$ is nontrivial as a smooth bundle (\S \ref{subsec-hdnonr}). We have thus reached a contradiction.

  Finally, in \Cref{sec:torsion}, we extend the singularity result to groups which may have torsion. To keep $\Gamma$
   torsion-free, we rather introduce a discrete cocompact subgroup $\widetilde \Gamma<\PSL_{2}(\R)$ (potentially with torsion), and denote by $\widetilde \mu$ a generating probability on $\widetilde \Gamma$ with finite support. 
  Invoking Selberg's theorem, we may consider a finite-index torsion-free  subgroup $\Gamma<\widetilde \Gamma$, and call   $\mu$ the measure induced by $\widetilde \mu$ on $\Gamma$. Note $\mu$ has a finite exponential moment (but not necessarily finite support), and the same Furstenberg measure as $\widetilde \mu$. We then observe that the proof of the torsion-free case applies verbatim to $\mu$, except the part where we show the rationality of $q\mapsto B_{\nu, \chi_{q}}$, as it relied on the finite support assumption. This part can however be checked for the specific measure $\mu$, using the fact that it is induced from a finitely supported measure. Hence the proof is complete.

\medskip

\paragraph{Independent work.} 
In a preprint posted on arXiv the day after the first
version of the present paper, Kosenko and Tiozzo~\cite{KT26}
independently established the singularity conjecture
under the additional assumption that $\supp\mu$
generates $\Gamma$ as a semigroup.
Their proof also focuses on the cocompact case. They rely on a different approach, comparing the Na\"im and Liouville geodesic currents
by means of Fourier analysis on the boundary. Building on Kosenko--Tiozzo's approach, Bogachev~\cite{Bogachev26}
subsequently extended their singularity result to all cocompact Kleinian groups
and certain higher-dimensional hyperbolic lattices.
In comparison with \cite{KT26}  and \cite{Bogachev26}, our proof is closer to the seminal work of Guivarc'h-Le Jan  \cite{GLJ90,GLJ93,GLJ96}, studying winding in the non-uniform case.

\paragraph{Conventions and notations.}
We write $\D$ for  the Poincaré disk, endowed with a basepoint $o$ which we fix once and for all.

All Hermitian products are linear in the first variable and
conjugate-linear in the second. The superscript ${}^{\mathsf T}$ denotes transposition.


\paragraph{Acknowledgements.} I am grateful to Serge Cantat and Yves Benoist for stimulating discussions and valuable comments on this work.

\paragraph{AI disclosure.} 
The proof combines three main ingredients: the author's earlier winding limit laws \cite{Benard}, Ratner's central limit theorem for the geodesic flow \cite{Ratner}, and the use of twisted cohomology to organize winding covariance identities across finite covers and ultimately derive a topological obstruction. Although these ingredients have antecedents in the literature, the idea of combining them to prove the singularity conjecture was proposed by ChatGPT (GPT-5.6). The argument was subsequently developed and refined through  further discussions with ChatGPT and  the author's own mathematical work, leading to the conceptual proof presented here. 

The present manuscript was written and organized by the author, with some assistance from ChatGPT in refining the exposition. 

The author has personally verified the mathematical arguments and references and takes full responsibility for the contents.

\section{Twisted cohomology to see all finite abelian covers at once}
\label{sec:twisted-overview}

In this section we fix a torsion-free cocompact discrete subgroup
$\Gamma<\PSL_2(\R)$. We write $S_{\Gamma}=\Gamma\backslash\D$ for the
associated closed hyperbolic surface and $g\ge2$ for its genus.
We use the standard presentation
\begin{equation}
\label{eq:tc-presentation}
 \Gamma=\left\langle a_1,b_1,\ldots,a_g,b_g\ \middle|\
                 \prod_{j=1}^g[a_j,b_j]=1\right\rangle,
 \qquad [a,b]=aba^{-1}b^{-1}.
\end{equation}
The cohomology of a finite abelian cover of $S_{\Gamma}$ splits into character
spaces for its deck group. We will realize these spaces as fibres of
one algebraic vector bundle, whose parameter records the character. On the unitary locus, harmonic representatives
provide a Hermitian metric and a Hodge-star operator on this bundle.

\subsection{Twisted cohomology}
\label{subsec:tc-cohomology}

Characters of $\Gamma$ are parametrized by $(\C^*)^{2g}$: to every 
$q=(u_1,v_1,\ldots,u_g,v_g)\in(\C^*)^{2g}$, one can associate the character
$\chi_q:\Gamma\to\C^*$ satisfying
\begin{equation}
\label{eq:tc-characters}
 \chi_q(a_i)=u_i,\qquad \chi_q(b_i)=v_i
 \qquad(1\le i\le g).
\end{equation}

Given $q\in (\C^*)^{2g}$, we define a $\chi_{q}$-equivariant de Rham cohomology. For $k\ge0$, let $\Omega^k(\D;\C)$ denote the space of smooth complex-valued $k$-forms on $\D$. Let $d : \Omega^k(\D;\C)\rightarrow \Omega^{k+1}(\D;\C)$ denote the differential operator. We consider the subspaces of $\chi_{q}$-equivariant forms
\[
 \Omega^k(\D;\C)_{\chi_q}
 =\{\omega\in\Omega^k(\D;\C):
           \gamma^*\omega=\chi_q(\gamma)\omega
           \text{ for every }\gamma\in\Gamma\}.
\]
For example,  in degree zero, $\Omega^{0}(\D;\C)=C^\infty(\D;\C)$ and equivariance means
$f(\gamma x)=\chi_q(\gamma)f(x)$. Since pullback commutes with $d$,
these spaces form a  complex of cochains. We will be interested in the first cohomology space
\begin{equation}
\label{eq:tc-de-rham}
 H_q^1
 =\frac{\{\omega\in\Omega^1(\D;\C)_{\chi_q}:d\omega=0\}}
             {d\Omega^0(\D;\C)_{\chi_q}}.
\end{equation}
Sometimes, we also write this space $ H_{\chi_{q}}^1$ to emphasize the role of the character. Note that for $q=1$,  we have $H^1_{1}= H^1_{\mathrm{dR}}(S_{\Gamma};\C)$, the  first de Rham cohomology space of $S_{\Gamma}$.

\begin{remark}\label{rem-linebdle}
$H^1_{q}$  also has an interpretation on $S_{\Gamma}$: Letting $\Gamma$ act on $\D\times \C$ 
by $\gamma\cdot(x,t)=(\gamma x,\chi_q(\gamma)t)$, 
the quotient $L_q=\Gamma\backslash (\D\times\C)$ defines a complex line bundle over $S_{\Gamma}$, and  $H_q^1$ corresponds to 
the first cohomology space of smooth $1$-forms on $S_{\Gamma}$ with values in $L_{q}$, i.e. $H^1_{q} \simeq H^1(S_{\Gamma}; L_q)$.
\end{remark}

Alternatively,  $H_q^1$ can be described by scalar cocycles. A
\emph{$\chi_q$-cocycle} is a function $\sigma:\Gamma\to\C$ such that
\[
 \sigma(\gamma_{1}\gamma_{2})=\sigma(\gamma_{1})+\chi_q(\gamma_{1})\sigma(\gamma_{2})
 \qquad(\gamma_{1},\gamma_{2}\in\Gamma).
\]
Write $Z^1_{\chi_q}(\Gamma;\C)$ for their vector space. The subspace of coboundaries is 
$ B^1_{\chi_q}(\Gamma;\C)
   =\{\gamma\mapsto(\chi_q(\gamma)-1)c:c\in\C\}$, and the associated cohomology space is
\[
 H^1_{\chi_q}(\Gamma;\C)
   =Z^1_{\chi_q}(\Gamma;\C)/B^1_{\chi_q}(\Gamma;\C).
\]
Given a closed form $\omega\in\Omega^1(\D;\C)_{\chi_q}$, we can define its associated $\chi_{q}$-cocycle by
\[
 \sigma_\omega(\gamma)=\int_{[o,\gamma o]}\omega.
\]
Here the integral is independent of the path in $\D$ because $\D$ is simply
connected and $\omega$ is closed. Moreover, the cocycle relation follows from the equivariance of $\omega$.

More generally, letting $f_{\omega}:\D\rightarrow \C$ denote the normalized primitive of $\omega$:
\[
 f_\omega(x)=\int_{[o,x]}\omega
 \]
we have for all $x\in \D, \gamma\in \Gamma$, 
\begin{equation}
\label{eq:tc-primitive}
 f_\omega(\gamma x)
       =\sigma_\omega(\gamma)+\chi_q(\gamma)f_\omega(x).
\end{equation}

\begin{lemma}[Period cocycles]
\label{lem:tc-periods}
For every $q\in(\C^*)^{2g}$, the assignment
$\omega\mapsto\sigma_\omega$ induces an isomorphism
\[
                   H_q^1\simeq H^1_{\chi_q}(\Gamma;\C).
\]
\end{lemma}
\begin{proof}
If $\omega=df$ with $f\in\Omega^0(\D;\C)_{\chi_q}$, then
$\sigma_\omega(\gamma)=(\chi_q(\gamma)-1)f(o)$, so the map is
well defined on cohomology. Conversely, if
$\sigma_\omega=(\chi_q-1)c$, formula \eqref{eq:tc-primitive} shows
that $f_\omega+c$ is $\chi_q$-equivariant. Thus $\omega$ is exact
in the complex \eqref{eq:tc-de-rham}, proving injectivity.

For surjectivity, choose $\rho\in C_c^\infty(\D;\R)$ with
$\sum_{\gamma\in\Gamma}\rho(\gamma^{-1}x)=1$. For example, such a function is
obtained by choosing a nonnegative compactly supported bump function
equal to one on a compact fundamental set, and dividing by the sum of
its translates (proper discontinuity makes this sum locally finite).
For $\sigma\in Z^1_{\chi_q}(\Gamma;\C)$, put
\[
                  f(x)=\sum_{\gamma\in\Gamma}
                              \rho(\gamma^{-1}x)\sigma(\gamma).
\]
Changing the summation variable gives
$f(\eta x)=\sigma(\eta)+\chi_q(\eta)f(x)$. Hence $df$ is a
closed $\chi_q$-equivariant form. Its normalized period cocycle is
\[
             \sigma_{df}(\eta)
               =\sigma(\eta)+(\chi_q(\eta)-1)f(o),
\]
which represents $[\sigma]$, as required.
\end{proof}

\subsection{Twisted cohomology bundle}
\label{subsec:tc-bundle}

Set $X=(\C^*)^{2g}\smallsetminus\{1\}$. We now give the family
$q\mapsto H_q^1$  the structure of an algebraic vector bundle over  $X$.
The following description in generator coordinates is the starting point. Recall the generators $a_{i},b_{i}$ of $\Gamma$ were introduced in \eqref{eq:tc-presentation}.

\begin{lemma}[Cocycle coordinates]
\label{lem:tc-coordinates}
Let $q=(u_1,v_1,\ldots,u_g,v_g)\in(\C^*)^{2g}$. Evaluation on the
generators, $\sigma\mapsto(x_1,y_1,\ldots,x_g,y_g)$ with
$x_i=\sigma(a_i)$ and $y_i=\sigma(b_i)$, induces the isomorphism
\begin{equation}
\label{eq:tc-cocycle-coordinates}
 Z^1_{\chi_q}(\Gamma;\C)\simeq
 \left\{(x_i,y_i)_{i=1}^g\in\C^{2g}:
       \sum_{i=1}^g\bigl((1-v_i)x_i+(u_i-1)y_i\bigr)=0\right\}.
\end{equation}
Under this identification, the coboundaries form the subspace
\begin{equation}
\label{eq:tc-coboundaries}
 B^1_{\chi_q}(\Gamma;\C)=\C b(q) \qquad \text{ where }\qquad
 b(q)=(u_1-1,v_1-1,\ldots,u_g-1,v_g-1).
\end{equation}
In particular, $\dim H_q^1=2g-2$ for $q\ne1$, whereas
$\dim H_1^1=2g$.
\end{lemma}

\begin{proof}
A map $\sigma :\Gamma \rightarrow \C$ is a $\chi_{q}$-equivariant cocycle if and only if the map 
\[
 \gamma\longmapsto
 \begin{pmatrix}\chi_q(\gamma)&\sigma(\gamma)\\0&1\end{pmatrix}
\]
is a group homomorphism from $\Gamma$ to $\mathrm{GL}_{2}(\C)$. In view of the presentation of $\Gamma$, defining a cocycle $\sigma$ therefore amounts to prescribing its values on the generators $a_{i},b_{i}$ ($1\leq i\leq g$) with the sole constraint that, under the cocycle rule, the relator has value zero: $\sigma(\prod_{j=1}^g[a_j,b_j])=0$.   
The cocycle identity gives
\[
 \sigma([a_i,b_i])=(1-v_i)\sigma(a_{i})+(u_i-1)\sigma(b_{i}).
\]
Since each commutator has character $1$, their product has cocycle value
equal to the sum of these expressions, so $\sigma(\prod_{j=1}^g[a_j,b_j])=0$ amounts to $\sum_{i}(1-v_i)\sigma(a_{i})+(u_i-1)\sigma(b_{i})=0$. This proves
\eqref{eq:tc-cocycle-coordinates}. 

Evaluating
$\gamma\mapsto(\chi_q(\gamma)-1)c$ proves
\eqref{eq:tc-coboundaries}.  For $q\ne1$, the relation is nonzero
and $b(q)\ne0$. As $H^1_{q}\simeq  Z^1_{\chi_q}(\Gamma;\C)/ B^1_{\chi_q}(\Gamma;\C)$ by \Cref{lem:tc-periods}, this gives dimension $(2g-1)-1$. At $q=1$, both the
relation and the coboundary space vanish, whence $\dim H^1_{1}=2g$.
\end{proof}

Recall that a function on a Zariski-open subset of $X$ is
\emph{regular} if, locally in the Zariski topology, it is a
quotient of polynomials in the coordinates $q$, with
nonvanishing denominator. An \emph{algebraic vector bundle}
$\mathcal V$ over $X$ is a vector bundle admitting
trivializations over a Zariski-open cover of $X$ whose
transition matrices have regular entries. A section
$s:X\to\mathcal V$ is called \emph{regular} if its coordinate
functions in these trivializations are regular.
For more details, see
\cite[Chapter~6, Section~1.2]{Shafarevich}.

For $q=(u_{i},v_{i})_{i=1}^g\in X$, write
\[
 \ell_q(x_1,y_1,\ldots,x_g,y_g)
    =\sum_{i=1}^g\bigl((1-v_i)x_i+(u_i-1)y_i\bigr).
\]
The coefficients of $\ell_q$ and $b(q)$ depend polynomially on $q$, and we have $\ell_q\neq0$, $b(q)\neq 0$ for every $q\in X$. Therefore, the families
\[
 \mathcal Z^1=\{(q,z)\in X\times\C^{2g}:\ell_q(z)=0\},
 \qquad
 \mathcal B^1=\{(q,c\,b(q)):q\in X,\ c\in\C\}
\]
define algebraic subbundles of $X\times \C^{2g}$, with respective ranks $2g-1$ and $1$. Moreover, 
$\mathcal B^1\subseteq  \mathcal Z^1$ because $\ell_q(b(q))=0$. 

 We define the \emph{Betti twisted cohomology bundle}
as the quotient algebraic bundle \cite[Chapter~6, Section~1.3]{Shafarevich}
\begin{equation}
\label{eq:tc-betti-bundle}
                        \mathcal H_X^1=\mathcal Z^1/\mathcal B^1.
\end{equation}
Its fibre at $q$ corresponds to $H_q^1$, by \Cref{lem:tc-periods,lem:tc-coordinates}. In particular the rank of $\mathcal H_X^1$ is  $2g-2$. 
The exclusion of $1$ is necessary because the dimension jumps there.

\bigskip
We now introduce explicit local trivializations of $ \mathcal H_X^1$. 
For $j\in \{1, \dots, g\}$, we set 
\[
 O_{u_j}=\{q\in X:u_j\ne1\},\qquad
 O_{v_j}=\{q\in X:v_j\ne1\}.
\]
For $q\in O_{u_j}$ and a cocycle $\sigma \in  Z^1_{\chi_q}(\Gamma;\C)$, put
\[
 \sigma^{(u_j)}(\gamma)
   =\sigma(\gamma)
      -\frac{\sigma(a_j)}{u_j-1}(\chi_q(\gamma)-1).
\]
This depends only on $[\sigma] \in H^1_{\chi_{q}}(\Gamma; \C)$ and is its unique representative
with $\sigma^{(u_j)}(a_j)=0$. Write
$A_i=\sigma^{(u_j)}(a_i)$ and $B_i=\sigma^{(u_j)}(b_i)$ for
the values of this normalized representative. The cocycle relation
is equivalent to
\begin{equation}
\label{eq:tc-u-chart}
 A_j=0,\qquad
 B_j=-\frac{1}{u_j-1}
          \sum_{i\ne j}\bigl((1-v_i)A_i+(u_i-1)B_i\bigr).
\end{equation}
Consequently,
\[
 \begin{split}
 \Phi_{u_j}:\mathcal H_X^1|_{O_{u_j}}
       &\longrightarrow O_{u_j}\times\C^{2g-2},\\
 (q,[\sigma])&\longmapsto\bigl(q,(A_i,B_i)_{i\ne j}\bigr)
 \end{split}
\]
is a fibrewise linear bijection. On $O_{v_j}$, use instead the unique representative
vanishing on $b_j$ to define $ \Phi_{v_j}$ similarly.

\begin{lemma}[Betti coordinates]
\label{lem:tc-charts}
The maps $\Phi_{u_j}, \Phi_{v_j}$ ($j=1, \dots, g$) form an atlas of local trivializations of the algebraic bundle $\mathcal H_X^1$. 
\end{lemma}

\begin{proof}
Let $j\in \{1, \dots, g\}$. Consider the algebraic subbundle  $\mathcal S_{j} \subseteq \mathcal Z^1_{|O_{u_{j}}}$  given by 
$$\mathcal S_{j}:=\{(q, z) \in O_{u_j}\times \C^{2g} \,:\, \ell_{q}(z)=0 \text{ and } x_{j}=0 \}$$
where implicitly, we write $z=(x_{1},y_{1}, \dots, x_{g},y_{g})$. 
We have the direct sum decomposition $\mathcal Z^1_{|O_{u_j}} =\mathcal S_{j}\oplus \mathcal{B}^1_{|O_{u_{j}}}$, therefore the projection map $\mathcal Z^1 \rightarrow \mathcal H^1_{X}$ gives an isomorphism of algebraic bundles 
$\mathcal S_{j} \simeq \mathcal H^1_{X |O_{u_j}}$. Seen on $\mathcal S_{j}$, the map $\Phi_{u_j}$ corresponds to forgetting the $j$-th coordinate pair $(x_{j},y_{j})$, in particular it is regular. As it is a fibrewise invertible regular bundle homomorphism, its inverse is also regular (by the inverse formula for a matrix in terms of the adjugate matrix). 
\end{proof}

\subsection{The unitary locus}
\label{subsec:tc-unitary}
Let
\[
                    U=(S^1)^{2g}\smallsetminus\{1\}\subset X
\]
denote the locus parametrizing nontrivial unitary characters of $\Gamma$. We see that each cohomology class above a point in $U$ has a so-called harmonic representative. We then use it to define a Hermitian structure on $\mathcal H^1_{U}:= \mathcal H^1_{X|U}$, the bundle obtained by restricting $\mathcal H^1_{X}$ to fibres above $U$. In this section, we will only see $\mathcal H^1_{U}$ as a smooth bundle. The regularity of the structures discussed here will be smooth only.

\bigskip

Let $x\in \D$. The space $T_{x}^*\D\otimes \C$ has a natural Hermitian\footnote{Our convention is that Hermitian inner products are $\C$-linear in the first variable, conjugate-$\C$-linear in the second variable.} structure, obtained by considering any positively oriented orthonormal basis $(e_{1},e_{2})$ of $T_{x}\D$ and claiming that the dual basis $(e^1, e^2)$ is orthonormal. 
We define the \emph{Hodge star} $\rot$ as the quarter-turn map on $T_{x}^*\D\otimes \C$, i.e. the linear map extending the relations
$$\rot e^1=e^2\quad \quad \quad \rot e^2=-e^1. $$
The above is independent of the choice of positively oriented orthonormal basis $(e_{1},e_{2})$.

The Hodge star naturally acts on the space of one-forms: for $\omega=(\omega_{x})_{x\in \D} \in  \Omega^1(\D;\C)$, we set $\rot \omega=(\rot \omega_{x})_{x\in \D}$. Note however, this action may not preserve closed forms, and in particular it does not pass to cohomology as yet. 

We say a  $1$-form $\omega \in \Omega^1(\D;\C)$ is \emph{co-closed} if $ \rot \omega$ is closed, and \emph{harmonic} if it is both closed and co-closed.

\begin{lemma}
\label{lem:hodge}
For $q\in U$, every cohomology  class $[\omega]\in H^1_{q}$ has a unique harmonic representative
$\omega^{\mathrm h}$.  Moreover, given any ball $O\subseteq \D$, the map $\mathcal H^1_{U}\rightarrow \Omega^1(O; \C), [\omega]\mapsto {\omega^{\mathrm h}}_{|O}$ is smooth.  
\end{lemma}

The uniqueness result 
can fail\footnote{Indeed, choose a non-zero  $\Gamma$-invariant holomorphic  one-form $\omega=a(z)dz$ on $\D$, set $\chi(\gamma)=\exp\circ f_{\omega}(\gamma o)$. Then the trivial class   $[0] \in H^1_{\chi}$ has at least two harmonic representatives:  $d( \exp \circ f_{\omega})$ and $0$.  Indeed, $\exp \circ f_{\omega}$ is a holomorphic $\chi$-equivariant function.}
  outside the unitary locus.
In the second claim, smoothness refers to the map $\mathcal{H}^1_{U}\times O\rightarrow T^*O\otimes_{\R}\C, ([\omega],x)\mapsto \omega^{\mathrm h}_x$ being smooth. 

\begin{proof}
Since $\chi_q$ is unitary, the associated flat line bundle
$L_q$ carries a parallel Hermitian metric. The Hodge theorem
for elliptic complexes \cite[Chapter~IV, Theorem~5.2]{Wells2008}
therefore gives a unique harmonic representative of every
class in $H_q^1$.
Moreover, the flat unitary bundles and their connections vary
smoothly with $q$, and $\dim H_q^1=2g-2$ is constant on $U$.
The claim on regularity follows from the  smooth-dependence theorem for harmonic projections
\cite[\S~1]{KodairaSpencer1960}.
\end{proof}

Harmonic representatives satisfy certain stability properties. 

\begin{lemma}
\label{lem:hodge2}
The Hodge-star $\rot$ and pullback by an
orientation-preserving isometry mutually commute, and send harmonic forms to harmonic forms.
\end{lemma}

In particular, \Cref{lem:hodge2} allows us to regard $\rot$ as acting on every $H^1_{q}$ ($q\in U$) via harmonic representatives: $\rot [\omega]=[\rot\omega^h]$. We write $\rot_{q}=\rot_{|H^1_{q}}$.

\begin{proof}
The commutation relation can be checked directly in local coordinates. Combined with the fact that $\omega\mapsto d\omega$ commutes with pullback by any smooth map, we obtain that the set of harmonic forms is preserved by orientation-preserving isometries. 
We finally check that $\rot$ preserves harmonic forms. Indeed, a form $\omega\in \Omega^1(\D, \C)$ is harmonic if and only if $d \omega=d \rot \omega=0$, equivalently $d(\rot\omega)=d(\rot^2\omega)=0$, since $\rot^2=-\Id$. This amounts to $\rot\omega$ being harmonic. 
\end{proof}

\begin{lemma} 
\label{lem:hodge-smooth}
The Hodge-star $\rot$ defines an automorphism of the smooth vector bundle $\mathcal{H}^1_{U}$.  
\end{lemma}

However, as we will see later, $\rot$ is \emph{not} an algebraic automorphism (in the sense of \Cref{prop-irrat-Hodge}). 

\begin{proof} Let $O\subseteq \D$ be a ball centered at $o$ large enough to contain a relatively compact fundamental domain for the action of $\Gamma$ on $\D$, as well as the segments $[o,a_i o]$ and $[o,b_i o]$ for $1\leq i\leq g$.
One can express $\rot$ as a composition of operations along the spaces $\mathcal H^1_{U}\rightarrow U\times \Omega^1(O; \C) \rightarrow U\times  \Omega^1(O; \C) \rightarrow \mathcal Z^1|_U \rightarrow \mathcal H^1_{U}$, given by $(q,[\omega])\mapsto (q,{\omega^{\mathrm  h}}_{|O}) \mapsto  (q,{\rot \omega^{\mathrm  h}}_{|O} )  \mapsto c:= \left(q, \, (\int_{[o, a_{i}o]} \rot \omega^{\mathrm  h}, \int_{[o, b_{i}o]} \rot \omega^{\mathrm  h})_{i}\right) \mapsto  c\mod \C b(q)$. All of these operations are smooth (notably thanks to \Cref{lem:hodge}), hence the claim.
\end{proof}

We use harmonic representatives to endow the restricted bundle $\mathcal{H}^1_{U}$ with a smooth Hermitian structure. Let $F$ be a fundamental domain for the action of $\Gamma$ on $\D$.
 Given $\alpha, \beta\in \Omega^1(\D, \C)_{\chi_{q}}$, set 
\[
  \langle\alpha,\beta\rangle_{L^2(S_{\Gamma})}
  =\int_{F} \langle\alpha_{x},\beta_{x}\rangle\,\dd A(x),
\]
where $\dd A$ is the hyperbolic area measure on $\D$, and the Hermitian product on $T_{x}^*\D\otimes \C$ has been defined above. The fact that $q$ is unitary and $\langle \cdot, \cdot\rangle$ is sesquilinear shows that $x\mapsto \langle\alpha_{x},\beta_{x}\rangle$ is $\Gamma$-invariant, so this does not depend on the choice of $F$.

\begin{lemma}
\label{lem:energy} Let $q\in U$.
For $\alpha, \beta\in \Omega^1(\D, \C)_{\chi_{q}}$, we have 
\begin{equation}
\label{eq:energy}
  \langle\alpha,\beta\rangle_{L^2(S_{\Gamma})}
  =\int_{F}\alpha\wedge\rot \overline \beta.
\end{equation}

\end{lemma}

\begin{proof}
Given $x\in \D$, and $e_{1},e_{2}$ a positively oriented orthonormal basis of $T_{x}\D$, the dual basis satisfies $e^1\wedge e^2=\dd A(x)$. Therefore, recalling $\rot e^1=e^2$, $\rot e^2=-e^1$, 
\[
  \alpha_{x}\wedge\rot\overline\beta_{x}=  ({\alpha_{x}(e_{1})}\overline{\beta_{x}(e_{1})}+{\alpha_{x}(e_{2})} \overline{\beta_{x}(e_{2})}) e^{1}\wedge e^2
  =\langle\alpha_{x},\beta_{x}\rangle\,\dd A(x)
\] 
and the claim on the inner product follows by integration over $F$. 

\end{proof}

\begin{lemma}
The Hermitian product $\langle \cdot , \cdot\rangle$  above defines a smooth Hermitian structure on the bundle $\mathcal H^1_{U}$ via the harmonic representatives: for every $q\in U$, $[\omega_{1}], [\omega_{2}] \in H^1_{q}$, 
$$ \langle [\omega_{1}], [\omega_{2}]\rangle_{L^2(S_{\Gamma})} = \langle \omega^{\mathrm h}_{1}, \, \omega^{\mathrm h}_{2} \rangle_{L^2(S_{\Gamma})}.$$
\end{lemma}

\begin{proof}
Indeed, smoothness follows from \Cref{lem:hodge}.
\end{proof}

\subsection{The role of torsion points}
\label{subsec:tc-torsion}

Let
\[
 \Utor=\{q\in U:q^n=1\text{ for some integer }n\ge1\}
\]
be the set of torsion points in $U$. It is dense and encodes nontrivial characters of $\Gamma$ with finite
image. 

For $q\in\Utor$, put
\[
 \Lambda_q=\ker\chi_q,\qquad
 S_{\Lambda_q}=\Lambda_q\backslash\D,\qquad
 d_q=[\Gamma:\Lambda_q].
\]
Then $S_{\Lambda_q}\to S_{\Gamma}$ is a connected cyclic cover of degree
$d_q=|\chi_q(\Gamma)|$. Every $\chi_{q}$-equivariant form $\omega\in \Omega^k(\D; \C)_{\chi_{q}}$ is in particular $\Lambda_q$-invariant, whence can be seen as a form on $ S_{\Lambda_q}$. This allows us to identify $\Omega^k(\D; \C)_{\chi_{q}}$ with $\Omega^k(S_{\Lambda_{q}}; \C)_{\chi_{q}}$, and the associated cohomology spaces.

Conversely, given an abelian finite cover  $S'\rightarrow S_{\Gamma}$, we can write $S'=\Gamma'\backslash \D$ where $\Gamma'\unlhd \Gamma$ is a normal subgroup such that $\Gamma/\Gamma'$ is finite abelian. Then, by elementary theory of representations of finite abelian groups \cite[Sections~2.6 and~3.1]{Serre1977},  the first cohomology of $S'$ splits over characters of $\Gamma/\Gamma'$: we have $H^1_{\mathrm{dR}}(S'; \C)=\oplus_{\chi :\Gamma/\Gamma' \rightarrow S^1} H^1_{\mathrm{dR}}(S'; \C)_{\chi}$. 
All the latter summands, more precisely those corresponding to nontrivial characters, appear in the bundle $\mathcal H^1_{X}$, as $H^1_{\mathrm{dR}}(S'; \C)_{\chi_{q}}\simeq H^1_{q}$.

In conclusion, the vector bundle $\mathcal H^1_{X}$ restricted to torsion points $\Utor$ of the unitary locus collects the cohomology of equivariant forms on nontrivial finite cyclic covers of $ S_{\Gamma}$. Moreover the (non-$\Gamma$-invariant) cohomology of any abelian finite cover can be recovered from those. 
The larger space $U$, or $X$,  supplies an algebraic framework for viewing all these cohomology spaces in one common structure.

\section{Winding statistics for random geodesic rays}

Let $\Gamma <\PSL_{2}(\R)$ be a cocompact torsion-free discrete subgroup, and let $\mu$ be a generating probability measure on $\Gamma$. We assume throughout that $\mu$ has finite exponential moment, i.e., for a word metric $|\cdot|$, there exists $\varepsilon>0$ such that $\E_{g\sim \mu}(e^{\varepsilon |g|})<\infty$. We denote by $\nu$ the Furstenberg measure on $\partial \D$ associated to $\mu$. 

For $\xi\in\partial\D$, let $r_\xi:[0,\infty)\to\D$ be the
unit-speed geodesic ray from $o$ to $\xi$. Given a closed $1$-form $\omega\in \Omega^1(\D;\C)$,  the quantity $f_\omega(r_\xi(t))=\int_{[o, r_{\xi}(t)]}\omega$ can be interpreted as the homological winding of the segment $[o, r_{\xi}(t)]$ observed in the direction $\omega$.
In this section, we study the  long-time statistics of $f_\omega(r_\xi(t))$ as the
endpoint $\xi$ varies according to $\nu$, and $\omega$ is assumed to be $\chi$-equivariant for some torsion character $\chi$ of $\Gamma$. We first record  the case of $\Gamma$-invariant forms (i.e. $\chi=1$), which dates back to the author's former work \cite{Benard}. Then we  infer the case of $\chi$-equivariant forms by using the induced random walk on $\ker \chi \unlhd \Gamma$. Put in the common framework of the twisted cohomology bundle from the previous section, we obtain that for every torsion point $q\in \Utor$, the limit covariance of the winding variables defines a Hermitian form on the twisted cohomology space $H^1_{q}$; or equivalently a bilinear form on $H^1_{q}\times H^1_{q^{-1}}$.  An important point is that, although this  form is naturally defined in terms of an induced walk, it actually has an explicit formula in terms of $\mu$. Finally, when $\nu=\leb$, we give alternative formulas for winding statistics which are due to Ratner and have a more geometric flavor.



\subsection{Statistics for \texorpdfstring{$\Gamma$}{Gamma}-invariant forms}
\label{sec:winding}

The \emph{hyperbolic drift} $L_{\mu}>0$ of the $\mu$-walk on $\Gamma$ is the unique constant such that, for a random variable $(\gamma_{i})_{i\geq1}\sim \mu^{\otimes \N^*}$, we have 
\[
             L_\mu=\lim_{n\to\infty}\frac{d(o,\gamma_{1}\dots \gamma_{n} o)}n
\]
 almost surely and in $L^1$, see for example \cite[Theorems~1.1--1.2]{MT}. 
The author's theorem
\cite[Theorem~1.4 and formula~(1)]{Benard} yields the following law of large numbers (LLN), central limit theorem (CLT), law of the iterated logarithm (LIL).

\begin{theorem}[B.] 
\label{thm:winding-laws}
Let $\omega\in \Omega^1(\D;\R)^\Gamma$ be a closed $\Gamma$-invariant real $1$-form on $\D$.
\begin{enumerate}[label=\textup{(\roman*)}]
\item[\emph{(LLN)}] Set $e_\nu(\omega)=\E_\mu \sigma_\omega/L_\mu$. For $\nu$-almost every $\xi$, we have
\[
       \frac{f_\omega(r_\xi(t))}{t}\longrightarrow e_\nu(\omega).
\]
\item[\emph{(CLT)}] If $e_\nu(\omega)=0$, put $V_\nu(\omega)=\E_\mu \sigma_\omega^2/L_\mu\ge0$.
Then, for $\xi$ varying according to $\nu$, we have
\[
       \frac{f_\omega(r_\xi(t))}{\sqrt t}
                    \xrightarrow{\mathrm{law}}\mathcal N(0,V_\nu(\omega)).
\]
\item[\emph{(LIL)}] If $e_\nu(\omega)=0$, then, for $\nu$-almost every $\xi$, we have
\begin{equation}
\label{eq:lil}
 \Acc_{t\to\infty}\frac{f_\omega(r_\xi(t))}{\sqrt{2t\log\log t}}
              =[-\sqrt{V_\nu(\omega)},\sqrt{V_\nu(\omega)}],
\end{equation}
where $\Acc$ denotes the set of accumulation points.
\end{enumerate}
\end{theorem}

As in the previous section, we  use the notation $\sigma_{\omega}(\gamma)=\int_{[o, \gamma o]}\omega$. Note here that the $\Gamma$-invariance of $\omega$ implies that $\sigma_{\omega}$ is a homomorphism from $\Gamma$ to $\C$. The notation
$\mathcal N(0,V_\nu(\omega))$ refers to the Gaussian distribution of variance $V_\nu(\omega)$. The variance may vanish, as it does for the lift of an exact form on $S_{\Gamma}$.

\begin{proof}
\Cref{thm:winding-laws} follows from \cite[Theorem~1.4 and formula~(1)]{Benard} which does apply because $\mu$ is  non-elementary \cite[Lemma~3.6]{CFFT}, and because $f_{\omega}:=\int_{[o, \cdot]}\omega$  is Lipschitz, $\Gamma$-equivariant with period homomorphism $\sigma_\omega$. 
\end{proof}

\subsection{Winding covariance for equivariant forms}
\label{sec-windcovform}

Let $\chi:\Gamma\rightarrow S^1$ be a nontrivial unitary character with finite image. The kernel $\Lambda=\ker \chi$ is a normal subgroup of $\Gamma$ with $\Gamma/\Lambda$ finite and cyclic. We see how the $\mu$-random walk on $\Gamma$ induces a probability measure $\mu_{\Lambda}$ on $\Lambda$, with the same Furstenberg measure $\nu$. 
By applying the previous subsection, this yields winding statistics for $\nu$-typical rays evaluated along $\chi$-equivariant forms. We then see how the mean and covariance appearing therein can be directly formulated in terms of $\mu$ (as opposed to $\mu_{\Lambda}$). Ultimately, the covariance part gives a Hermitian form on $H^1_{q}$ for every torsion point $q\in \Utor$.

\bigskip
Let us first define the induced measure on $\Lambda$. Let $(Y_{i})_{i\geq 1}$ be i.i.d. variables of law $\mu$, write $X_n=Y_1\cdots Y_n$. Let
\[
 \tau=\min\{n\ge1:X_n\in \Lambda\},\qquad
 \mu_{\Lambda}=\text{the law of }X_\tau.
\]

\begin{lemma}
\label{lem:returns}
The return time $\tau$ has a finite exponential moment, and its expectation is $\E\tau=[\Gamma:\Lambda]=:m$.
The measure $\mu_\Lambda$ has a finite exponential  moment, its support
generates $\Lambda$ as a group, and its hitting measure is $\nu$. Moreover,
\begin{equation}
\label{eq:drift-return}
                         L_{\mu_\Lambda}= m L_\mu.
\end{equation}
\end{lemma}

\begin{remark}\label{rem-return}
Although the current setting imposes various restrictions on $\Gamma$ and $\Lambda$, this lemma is valid as soon as $\Gamma$ is a Zariski-dense discrete subgroup of $\PSL_{2}(\R)$, and $\Lambda\leq \Gamma$ is a finite-index subgroup, provided $\mu$ is generating and has a finite exponential moment. Indeed the proof only uses those properties.
\end{remark}

\begin{proof}
Let $D=\Lambda \backslash  \Gamma$. The projected walk on $D$ is irreducible. Since
$D$ is finite, there exist $k\geq1$ and $c>0$ such that, from every
state, the probability of hitting the base coset $\Lambda$ within $k$ steps is at least $c$.
Hence, by the Markov property,
\[
 \Prob(\tau>jk)\leq(1-c)^j,
\]
so $\tau$ has an exponential moment. The uniform probability on $D$ is
stationary; Kac's return-time formula therefore gives
\[
                         \E\tau=|D|=m.
\]

The finite exponential moment for $\mu_\Lambda$ comes from that of $\mu$ and $\tau$.
Let us check $\mu_{\Lambda}$ generates $\Lambda$.  For each $\delta\in D$, choose  $t_\delta,r_\delta \in \Gamma$ which are  products of elements  in $\{e\}\cup \supp\mu$,  and whose projections to $D$ satisfy $[e]  t_\delta=\delta$, $\delta r_{\delta}=[e]$, with $t_{[e]}=e$.  The elements
$t_\delta s t_{\delta s}^{-1}$ where $s\in \supp \mu$, $\delta\in D$ generate $\Lambda$ (by the Schreier
rewriting argument). 
Moreover, writing
\[
 t_\delta s t_{\delta s}^{-1}
   =(t_\delta s r_{\delta s})(t_{\delta s}r_{\delta s})^{-1}
\]
and decomposing each positive loop at its successive returns to the base coset $\Lambda$, we see those generators belong to $\langle \supp \mu_{\Lambda}\rangle_{\mathrm{grp}}$. Hence $\mu_\Lambda$ generates $\Lambda$.

Let $\tau_j$ be the $j$-th return time to $\Lambda$. By the strong
Markov property, $(X_{\tau_j})_{j\geq1}$ is the $\mu_\Lambda$-random
walk, while $X_{\tau_j}o$ is a subsequence of $X_no$. Hence its hitting
measure is $\nu$. Moreover, the increments $\tau_j-\tau_{j-1}$ are i.i.d. with mean $m$, so
\[
 \frac{\tau_j}{j}\longrightarrow \E\tau=m \qquad\text{a.s.}
\]
Therefore
\[
 \frac{d(o,X_{\tau_j}o)}{j}
 =\frac{d(o,X_{\tau_j}o)}{\tau_j}\frac{\tau_j}{j}
 \longrightarrow mL_\mu,
\]
which is \eqref{eq:drift-return}.
\end{proof}

In view of \Cref{lem:returns}, we may apply \Cref{thm:winding-laws} to the measure $\mu_{\Lambda}$ to get winding statistics for $\chi$-equivariant forms. More precisely, applying \Cref{thm:winding-laws} to both the real and imaginary parts of  a closed form $\omega\in \Omega^1(\D;\C)_{\chi}$, the mean and variance appearing therein give access to
$$e_{\nu}(\omega):= L_{\mu_{\Lambda}}^{-1}\E_{\mu_{\Lambda}}(\sigma_{\omega})    \qquad \qquad V_{\nu}(\omega):= L_{\mu_{\Lambda}}^{-1} \E_{\mu_{\Lambda}}(|\sigma_{\omega}|^2).$$
Note that those terms only depend on the cohomology class $[\omega]\in H^1_{\chi}$: indeed, any two representatives $\omega, \omega'$ satisfy $\sigma_{\omega|\Lambda}= \sigma_{\omega'|\Lambda}$. In particular, the variance term defines a Hermitian form on $H^1_{\chi}$. To remain in the setting of complex vector bundles (as opposed to real vector bundles), we prefer to see  this Hermitian form as a $\C$-bilinear pairing between $H^1_{\chi}$ and its conjugate $H^1_{\chi^{-1}}$. The pairing is  explicitly $B_{\nu, \chi}: H^1_{\chi} \times H^1_{\chi^{-1}} \rightarrow \C$,  
\begin{equation}\label{eq-covpa}
B_{\nu, \chi}( [\omega], [\zeta]):=L_{\mu_{\Lambda}}^{-1} \E_{\mu_{\Lambda}}(  \sigma_{\omega}\sigma_{\zeta}). 
\end{equation}
We call it the \emph{winding covariance pairing}.
It will be useful to relate the above quantities to $\mu$ directly. For this, we first observe that the cocycle $\sigma_{\omega}$ is cohomologous to a unique cocycle $\sigma^0_{\omega}$ with zero $\mu$-average. 

\begin{lemma}
\label{lem:centered}
For unitary $\chi\ne1$, every class in $H^1_\chi(\Gamma;\C)$ has a unique
\emph{centered} representative $\sigma^0$, meaning $\E_\mu\sigma^0=0$.
It is given by
\begin{equation}
\label{eq:centered}
 \sigma^0(\gamma)=\sigma(\gamma)
       +(\chi(\gamma)-1)\frac{\E_\mu\sigma}{1-\E_\mu\chi}.
\end{equation}
\end{lemma}
\begin{proof}
An average of unit complex numbers can equal $1$ only when each number
with positive weight equals $1$. Thus $\E_\mu\chi=1$ would make $\chi$
trivial on $\supp\mu$, hence on $\Gamma$.
The denominator is nonzero, and solving
$\E_\mu[\sigma+(\chi-1)a]=0$ gives the displayed formula and uniqueness.
\end{proof}

\begin{remark} \label{rem-cent-coc}
i) For the trivial character,  coboundaries
vanish, so centering cannot be imposed by changing the representative.

ii) For a possibly nonunitary nontrivial character $\chi:\Gamma\to\C^*$, the same statement and proof apply provided the relevant expectations are finite and $\E_\mu\chi\ne1$.

iii) The harmonic form  and the centered cocycle representatives need not correspond under
\cref{lem:tc-periods}: the former is related to the metric,
whereas the latter depends on $\mu$.
\end{remark}

\begin{lemma}
\label{lem:stopped-covariance}
Assume that $\chi\ne1$ is a unitary character with finite image. Let $\omega \in \Omega^1(\D;\C)_{\chi}$ be a
$\chi$-equivariant closed $1$-form. Then
\begin{equation}
\label{eq:stopped-moments}
 \E_{\mu_\Lambda}(\sigma_{\omega})=  0 \qquad \qquad  \E_{\mu_\Lambda}(|\sigma_{\omega}|^2)
               =m \E_\mu (|\sigma^0_{\omega}|^2).
\end{equation}
In particular, for $[\omega]\in H^1_{\chi}$,  $[\zeta]\in H^1_{\chi^{-1}}$, 
\begin{equation}
\label{eq:covariance-centered}
B_{\nu, \chi}([\omega], [\zeta])= L_{\mu}^{-1} \E_\mu (\sigma^0_{\omega} \sigma^0_{\zeta}).
\end{equation}

\end{lemma}

\begin{proof}
Set $M_n=\sigma_\omega^0(X_n)$. By the cocycle identity,
\[
 M_{n+1}-M_n=\chi(X_n)\sigma_\omega^0(Y_{n+1}).
\]
Since $\E_\mu\sigma_\omega^0=0$, $(M_n)$ is a complex martingale. Put
$K=\E_\mu|\sigma_\omega^0|^2$. 
Consider integers \(r\ge s\ge 0\). Writing \(\Delta M_{k+1}:=M_{k+1}-M_k\), we have
\[
M_{\tau\wedge r}-M_{\tau\wedge s}
=
\sum_{k=s}^{r-1}\mathbf 1_{\{\tau>k\}}\Delta M_{k+1}.
\]
Since \(\{\tau>k\}\) is $(Y_{1}, \dots, Y_{k})$-measurable, the above increments are orthogonal in \(L^2\). Therefore 
\[
\begin{aligned}
\E\left|M_{\tau\wedge r}-M_{\tau\wedge s}\right|^2
&=
\sum_{k=s}^{r-1}
\E\!\left[
\mathbf 1_{\{\tau>k\}}
|\Delta M_{k+1}|^2
\right]\\
&=
\sum_{k=s}^{r-1}
\E\!\left[
\mathbf 1_{\{\tau>k\}}
\E\!\left(|\Delta M_{k+1}|^2\mid Y_1,\ldots,Y_k\right)
\right]\\
&=
K\sum_{k=s}^{r-1}\,\E\!\left[
\mathbf 1_{\{\tau>k\}}
\right]\\
&=
K\,\E\bigl[(\tau\wedge r)-(\tau\wedge s)\bigr].
\end{aligned}
\]
As $\E\tau=m<\infty$, we get that $M_{\tau\wedge n}\to M_\tau$ in $L^2$. Optional
stopping at $\tau\wedge n$ and passage to the limit give
\[
                         \E M_\tau=0.
\]
The cocycles $\sigma_\omega$ and $\sigma_\omega^0$ differ by a
coboundary $(\chi-1)c$, which vanishes on $\Lambda$; hence
$M_\tau=\sigma_\omega(X_\tau)$. Thus
\[
                         \E_{\mu_\Lambda}\sigma_\omega=0,
\]
which gives the first identity in \eqref{eq:stopped-moments}.
Finally, orthogonality of the martingale increments gives in the same way as before
\[
 \E|M_{\tau\wedge n}|^2=K\,\E(\tau\wedge n).
\]
Letting $n\to\infty$ yields
\[
 \E_{\mu_\Lambda}|\sigma_\omega|^2
 =\E|M_\tau|^2
 =mK
 =m\E_\mu|\sigma_\omega^0|^2.
\]

The final claim about $B_{\nu,\chi}$ follows by polarization from the combination of \eqref{eq:drift-return} and \eqref{eq:stopped-moments}.
\end{proof}

\subsection{Winding for rays chosen by \texorpdfstring{$\leb$}{leb}}
\label{subsec:geometric}

The winding estimates from the previous sections apply for the angular measure $\leb$ on $\partial \D$. Indeed, it is the Furstenberg measure of some generating measure on $\Gamma$ with finite exponential moment, see J. Li \cite[Appendix~A, Theorem~A.1]{LNP} (which applies thanks to \cite[Corollary~4.15]{Quint}). We see below that for the angular measure, we can in fact describe geometrically the winding covariance form $B_{\leb,\chi}$. The proof relies on Ratner's CLT from  \cite{Ratner}.

\bigskip
The key proposition, formulated in the spirit of \Cref{thm:winding-laws}, is the following. 

\begin{proposition}
\label{prop:geometric}
For a real closed $\Gamma$-invariant one-form $\omega \in \Omega^1(\D, \R)^\Gamma$,
angular-almost-every ray has zero winding drift, i.e. $e_{\leb}(\omega)=0$. The angular ray CLT
and LIL have variance
\begin{equation}
\label{eq:geometric-variance}
 V_{\leb}(\omega)=\frac{2}{\Area(S_{\Gamma})}
                              \norm{\omega^{\mathrm h}}_{L^2(S_{\Gamma})}^2.
\end{equation}
\end{proposition}

By polarization, we obtain the following formula for the bilinear covariance pairing on $H^1_{\chi}\times H^1_{\chi^{-1}}$.
\begin{corollary}
Let $F\subseteq \D$ be a fundamental domain for the action of $\Gamma$. For every non-trivial character $\chi :\Gamma \rightarrow S^1$ with finite image, for $[\omega]\in H^1_{\chi}$, $ [\zeta]\in H^1_{\chi^{-1}}$, 
\label{cor:geometric}
\begin{equation}
\label{eq:geometric-variance-form}
 B_{\leb, \chi}([\omega], [\zeta])=\frac{2}{\Area(S_{\Gamma})}
                              \int_{F} \omega^{\mathrm h} \wedge \rot \zeta^{\mathrm h}.
\end{equation}
\end{corollary}

\begin{proof}[Proof of \Cref{cor:geometric}]
Every $\chi$-equivariant form $\omega\in \Omega^1(\D, \C)_{\chi}$ is  invariant under $\Lambda=\ker \chi$.  Apply \Cref{prop:geometric} with $\Lambda$ in place of $\Gamma$ to the real and imaginary parts of $\omega$ to get 
$ V_{\leb}(\omega)=\frac{2}{\Area(S_{\Lambda})}
                              \norm{\omega^{\mathrm h}}_{L^2(S_{\Lambda})}^2=\frac{2}{\Area(S_{\Gamma})}
                              \norm{\omega^{\mathrm h}}_{L^2(S_{\Gamma})}^2.$ Here the last equality uses that $\D \rightarrow \C, x\mapsto \langle \omega^{\mathrm h}_{x},  \omega^{\mathrm h}_{x}\rangle$ is not only $\Lambda$-invariant, but $\Gamma$-invariant. Recalling how $B_{\leb, \chi}$ was defined and polarizing, it follows that $ B_{\leb, \chi}([\omega], [\zeta])=\frac{2}{\Area(S_{\Gamma})}
                              \langle \omega^{\mathrm h}, \overline{\zeta^{\mathrm h}}\rangle_{L^2(S_{\Gamma})}$, whence the result in view of  \Cref{lem:energy}.
\end{proof}

We now verify the key proposition. 

\begin{proof}[Proof of \Cref{prop:geometric}]
As noted earlier, $V_{\leb}(\omega)$ only depends on the cohomology class $[\omega]$, so we may assume $\omega$ to be harmonic.

Let $F\subseteq \D$ be a  fundamental domain for the action of $\Gamma$. Write $\dd\ell$ normalized Liouville measure on $T^1F$, in other terms, in polar coordinates 
\[
 \dd\ell=\frac{\dd A(x)_{|F}}{\Area(F)}\frac{\dd\theta}{2\pi}.
\]
Let $g_t$ denote the time-$t$  geodesic flow on $T^1S_{\Gamma}$ and, with the same notation, its lift to $T^1\D$. 
For $v\in T^1F$, let
\[
 W_T(v)=\int_0^T \omega(g_t v)\dd t
\]
denote its $\omega$-winding up to time $T$.
Ratner \cite{Ratner} gives a central limit theorem for $W_T(v)$ as $v$ varies according to $\dd\ell$:
\begin{equation}\label{eqRatC}
 \frac{W_T}{\sqrt T}
 \,\xrightarrow[T\to +\infty]{\mathrm{law}}  \, \mathcal N(0,\sigma^2),
\end{equation}
where $ \sigma^2
 :=\lim_{T\to\infty}\frac1T\int_{T^1F} W_T^2\dd\ell$.

By comparing \eqref{eqRatC} with the LLN in \Cref{thm:winding-laws}, we get that $e_{\leb}(\omega)=0$. 
Moreover,  given any point $x\in \D$,  let $\lambda_x$ be the hitting distribution on  $\partial \D$ for a random geodesic ray in $\D$ originating from $x$ and with direction chosen by $\leb$ on $T^1_{x}\D$. This is the normalized Patterson--Sullivan measure
based at $x$. Hence \cite[Appendix~A, Theorem~A.1]{LNP} realizes $\lambda_x$ as the Furstenberg measure of a generating random walk on $\Gamma$
with finite exponential moment. Applying the CLT in \Cref{thm:winding-laws} to this random walk, we obtain a limiting variance, and by the LIL  in \Cref{thm:winding-laws}, the latter only depends on the measure class of $\lambda_{x}$, so it must be equal to $V_{\leb}(\omega)$. From this and \eqref{eqRatC}, we deduce that 
$$ V_{\leb}(\omega)=\sigma^2.$$

It remains to identify $\sigma^2$. For that, set
\begin{equation}
\label{eq:correlation}
 C_{\omega}(t)
 :=\int_{T^1F}\overline{\omega(v)}\omega(g_t v)\dd\ell(v).
\end{equation}
We first compute the angular average above the origin of the disk $x=0$.
Writing $\omega=A\,dz+B\,d\bar z$, the equations
$d\omega=d\rot\omega=0$ are equivalent to
$\partial_{\bar z}A=\partial_zB=0$. Thus
\[
 \omega=a(z)\,dz+b(\bar z)\,d\bar z,
\]
where $a,b:\D\to\C$ are holomorphic.
Recall the hyperbolic metric on $\D$ is $
 \frac{4|\dd z|^2}{(1-|z|^2)^2}$, in particular the unit vector of angle $\theta$ based at $0$ is $ v_\theta=\frac{e^{i\theta}}2$, and  for $t\geq0$, putting $r=\tanh(t/2)$, we have $ g_tv_\theta= \frac{1-r^2}{2}e^{i\theta}$ seen as a tangent vector based at $re^{i\theta}$.

Considering the Taylor expansions \[
 a(z)=\sum_{n\geq0}a_nz^n,
 \qquad
 b(\bar z)=\sum_{n\geq0}b_n\bar z^n,
\]
we get
\[
\begin{aligned}
 \omega(v_\theta)&=\frac12\left(a_0e^{i\theta}+b_0e^{-i\theta}\right),\\
 \omega(g_tv_\theta)&=\frac{1-r^2}{2}
 \left(a(re^{i\theta})e^{i\theta}+b(re^{-i\theta})e^{-i\theta}\right).
\end{aligned}
\]
Hence
\[
\overline{\omega(v_\theta)}
\omega(g_tv_\theta)=
\frac{1-r^2}{4}\Big(
 \overline{a_0}a(re^{i\theta})
 +\overline{a_0}b(re^{-i\theta})e^{-2i\theta}
 +\overline{b_0}a(re^{i\theta})e^{2i\theta}
 +\overline{b_0}b(re^{-i\theta})
\Big).
\]
 Integration in $\theta$ kills every nonconstant Fourier mode, leading to 
 \[
 \int_0^{2\pi}
 \overline{\omega(v_\theta)}
 \omega(g_tv_\theta)
 \frac{\dd\theta}{2\pi}
 =
 \frac{1-r^2}{4}\left(|a_0|^2+|b_0|^2\right)
= \frac{1-r^2}{2} \|\omega_0\|^2.
\]
Here we assumed $x=0$, but the obtained equality extends to any $x\in \D$ by pushing $\omega$ around via orientation-preserving hyperbolic isometries of $\D$ (which preserve harmonic forms by \Cref{lem:hodge2}). Integrating over $x\in F$
 and using
$ 1-r^2=1-\tanh^2(t/2)=\sech^2(t/2)$
gives
\[
 C_\omega(t)
 =
 \frac{\|\omega\|_{L^2(S_\Gamma)}^2}{2\Area(S_{\Gamma})}
 \sech^2(t/2).
\]

We derive that
\begin{align*}
 \sigma^2
 &=\lim_{T\to\infty}
   2\int_0^T\left(1-\frac tT\right)C_{\omega}(t)\dd t\\
 &=\frac{\norm\omega_{L^2(S_{\Gamma})}^2}{\Area(S_{\Gamma})}
      \int_0^\infty\sech^2(t/2)\dd t
 =\frac{2\norm\omega_{L^2(S_{\Gamma})}^2}{\Area(S_{\Gamma})}.
\end{align*}
This finishes the proof.
\end{proof}

\section{Topological obstruction to absolute continuity}

Keep the notation $\Gamma <\PSL_{2}(\R)$ to denote a cocompact torsion-free discrete subgroup.  Let $\mu$ be a generating probability measure on $\Gamma$ with finite exponential moment. Set $\nu$ the associated Furstenberg measure on $\partial \D$. In this section, we show that if $\mu$ has finite support, then $\nu$ and $\leb$ must be mutually singular. 

The argument goes as follows: Recall the twisted cohomology bundle $\mathcal H^1_{X}$ over $X=(\C^*)^{2g}\smallsetminus \{1\}$, and the unitary locus $U$. For every $q\in \Utor$, recall $B_{\nu, \chi_{q}}:H^1_{q}\times H^1_{q^{-1}}\rightarrow \C$  denotes the winding covariance bilinear pairing, defined in  the previous section. Assuming that $\mu$ has finite support, we show that the field of bilinear pairings  $q\mapsto B_{\nu,\chi_{q}}$ extends continuously from $\Utor$ to a  Zariski-open neighborhood of $U$ in $X$, with extension depending on $q$ regularly for the algebraic bundle structure. Then, assuming $\nu \not \perp \leb$, we deduce the same holds for the Hodge-star assignment $q\mapsto \rot_{q}$, originally defined on $U$. Finally, we show the Hodge star cannot have such a rational extension, because it would force the subbundle of $(1,0)$-forms $\mathcal H^{(1,0)}_{U}$ (defined below) to have trivial determinant, while it is known  not to be the case. 

\subsection{From finite support to rational extension of the covariance form}
\label{subsec:finisuprat}

Let $\iota : X\rightarrow X$ be the involution $\iota(q)=q^{-1}$. Using $\vee$ to indicate the dual space or bundle, we introduce the bundle of bilinear pairings 
$$\Bil(\mathcal H^1_{X}, \iota^*\mathcal H^1_{X})= (\mathcal H^1_{X})^{\vee}\otimes (\iota^*\mathcal H^1_{X})^{\vee}.$$
It is the vector bundle over $X$ whose fibre above $q$ is $(H^1_{q})^\vee \otimes (H^1_{q^{-1}})^\vee$, i.e. consists of the vector space of $\C$-bilinear forms $H^1_{q}\times H^1_{q^{-1}}\rightarrow \C$. Local trivializations of $\mathcal H^1_{X}$ give $\Bil(\mathcal H^1_{X}, \iota^*\mathcal H^1_{X})$ the structure of an algebraic bundle. 

\begin{lemma}\label{lem-finite-supp-ratext}
If $\mu$ has finite support, then the assignment $q\mapsto B_{\nu, \chi_{q}}$ defined on $\Utor$ extends to a regular section of the restricted bundle $\Bil(\mathcal H^1_{X}, \iota^*\mathcal H^1_{X})_{|O}$ over a Zariski-open subset $O\subseteq X$ containing $U$ and stable by $q\mapsto q^{-1}$. 
\end{lemma}

\begin{proof}
Set
\[
 D(q):=1-\E_\mu(\chi_q)
      =1-\sum_{\gamma\in\supp\mu}\mu(\gamma)\chi_q(\gamma).
\] 
Since $\mu$ has finite support and, for every $\gamma\in\Gamma$, the map
$q\mapsto\chi_q(\gamma)$ is a Laurent monomial, the map $D$ must be  regular on
$(\C^*)^{2g}$. Put
\[
 O:=\{q\in X:D(q)D(q^{-1})\ne0\}.
\]
This is a  Zariski-open subset of $X$. Moreover, $U\subseteq O$, as noted already in \Cref{lem:centered}.

Let $q\in O$. Recall we have identified $H_q^1\equiv H^1_{\chi_q}(\Gamma;\C)$ via \Cref{lem:tc-periods}. Below, we freely use this identification. Given a cocycle class $[\sigma]\in H_q^1$, \Cref{lem:centered}, and \cref{rem-cent-coc} ii), have characterized the unique $\mu$-centered representative of $[\sigma]$ as
\[
 \sigma^0(\gamma)
 :=
 \sigma(\gamma)
 +(\chi_q(\gamma)-1)\frac{\E_\mu\sigma}{D(q)}.
\]
For any $\gamma\in \Gamma$, the assignment
\[
 q\longmapsto  ([\sigma]\mapsto \sigma^0(\gamma))
\]
is a regular section of the dual bundle $(\mathcal{H}^1_{X})^\vee_{|O}$, as can be seen directly in the Betti coordinates from \Cref{lem:tc-charts} and the cocycle formula.
Therefore, for $q\in O$, $[\sigma]\in H_q^1$ and $[\psi]\in H_{q^{-1}}^1$, the pairing
\[
 \widetilde B_q([\sigma],[\psi])
 :=
 \frac1{L_\mu}
 \sum_{\gamma\in\supp\mu}
 \mu(\gamma)\,
  \sigma^0(\gamma)\,
 \psi^0(\gamma)
\]
defines  a regular section of
$ \Bil(\mathcal H_X^1,\iota^*\mathcal H_X^1)_{|O}.$

On the other hand, \Cref{lem:stopped-covariance} gives, for $q\in\Utor$,
\[
 \widetilde B_q=B_{\nu,\chi_q}.
\]
Thus $\widetilde B$ is the required extension.

\end{proof}

\subsection{Absolute continuity forces the Hodge star to be rational in Betti coordinates}
\label{subsec-rathodsbet}

 Recall from \Cref{lem:hodge-smooth} that the Hodge-star defines   a certain smooth section of $\mathrm{End}(\mathcal H^1_{X})_{|U}$, that we denote by $q\mapsto \rot_{q}$. Keep $O$ as in \Cref{lem-finite-supp-ratext}. 
The goal of the subsection is to show the following.

\begin{lemma}
\label{finite-notsing-rat-HS}
 If $\mu$ has finite support, and $\nu \not \perp \leb$, then the Hodge-star assignment $q\mapsto \rot_{q}$ (defined on $U$) extends to a regular section of the algebraic bundle of endomorphisms $\mathrm{End}(\mathcal H^1_{X})_{|O}$. 
\end{lemma}

\begin{proof}[Proof of \Cref{finite-notsing-rat-HS}]

For $q\in X$, let $I_{q}: H^1_{q}\times H^1_{q^{-1}}\rightarrow \C$ denote the intersection pairing,  given by
$$I_{q}([\omega], [\zeta])=\int_{F} \omega \wedge \zeta $$
where $F$ stands for some fundamental domain of the action of $\Gamma$ on $\D$. Poincaré duality with local coefficients \cite[Sections~1.4 and~1.8]{Goldman} 
guarantees the pairing $I_{q}$ is non-degenerate. By \Cref{cor:geometric}, we have for any $q\in \Utor$, 
\begin{equation}\label{BvsI}
B_{\leb, \chi_{q}}(\cdot, \cdot)= \frac{2}{\Area(F)}I_{q}(\cdot, \rot_{q^{-1}}\cdot).
\end{equation}
By the mean cancellation $e_{\leb}=0$ from \Cref{prop:geometric} and the law of the iterated logarithm from \Cref{thm:winding-laws}, the assumption $\nu \not \perp \leb$ forces $B_{\leb, \chi_{q}}=B_{\nu, \chi_{q}}$. The finite support of $\mu$ then allows to apply \Cref{lem-finite-supp-ratext} to justify that $q\mapsto B_{\leb, \chi_{q}}$ extends regularly to $O$; we write $q\mapsto \widetilde B_q$ this extension. We deduce from \eqref{BvsI} and the non-degeneracy of $I_{q}$ an extension of $q\mapsto \rot_{q}$ to $O$. More precisely, for $q\in O$, set $I_q^\flat,\widetilde B_q^\flat:
H^1_{q^{-1}}\longrightarrow (H^1_q)^\vee$ to be the linear maps given by 
\[
I_q^\flat(v)=I_q(\,\cdot\,,v),
\qquad
\widetilde B_q^\flat(v)=\widetilde B_q(\,\cdot\,,v).
\]
By \eqref{BvsI} and non-degeneracy of $I_{q}$, we have for $q\in \Utor$, 
\[
\rot_{q^{-1}}
=
\frac{\Area(S_\Gamma)}{2}
(I_q^\flat)^{-1}\circ\widetilde B_q^\flat.
\]
This formula provides an extension of $q\mapsto \rot_{q}$ to $O$ (recall $O$ is invariant under $q\mapsto q^{-1}$).  
It remains to check that  this extension is regular in the parameter $q$. For that, we only need to verify that $q\mapsto I_{q}$ is regular on $O$. We check this in \Cref{Iqrati} below.
\end{proof}

\begin{lemma}\label{Iqrati}
The assignment $q\mapsto I_{q}$ defines a regular section of the algebraic bundle of bilinear pairings $\Bil(\mathcal H^1_{X}, \iota^*\mathcal H^1_{X})$.
\end{lemma}

\begin{proof}

Since $\omega\wedge\zeta$ is $\Gamma$-invariant, the integral
defining $I_q$ is independent of the choice of fundamental
domain $F$. We may therefore choose a \emph{polygonal fundamental
domain} $F\subseteq\D$, obtained by lifting the unique
$2$-cell in the standard one-vertex cell decomposition of
$S_\Gamma$, with vertex $\pi(o)$ and $2g$ edges.
More precisely, $F$ is a closed topological disk with piecewise
smooth boundary consisting of $4g$ oriented sides, paired as
$(e_j,e_j')_{j=1}^{2g}$, such that
$e_j'=-\gamma_j e_j$ for some $\gamma_j\in\Gamma.$
Here the minus sign denotes reversal of orientation.
Moreover, every vertex of $F$ belongs to $\Gamma o$.

Fix $q\in X$, and let
$\omega\in\Omega^1(\D;\C)_{\chi_q}$ and
$\zeta\in\Omega^1(\D;\C)_{\chi_{q^{-1}}}$
be closed. Let $f_\omega,f_\zeta$ be their normalized
primitives, and $\sigma_\omega,\sigma_\zeta$ their
period cocycles, as in \eqref{eq:tc-primitive}.
Since $\omega\wedge\zeta=d(f_\omega\zeta)$,
Stokes' theorem gives
\[
 I_q([\omega],[\zeta])
 =\int_{\partial F}f_\omega\zeta.
\]
Using
\[
 f_\omega(\gamma x)
 =\sigma_\omega(\gamma)+\chi_q(\gamma)f_\omega(x),
 \qquad
 \gamma^*\zeta=\chi_q(\gamma)^{-1}\zeta,
\]
we obtain
\[
\begin{aligned}
\int_{e_j}f_\omega\zeta+
\int_{e_j'}f_\omega\zeta
&=\int_{e_j}
\bigl(f_\omega\zeta-\gamma_j^*(f_\omega\zeta)\bigr)\\
&=-\chi_q(\gamma_j)^{-1}
  \sigma_\omega(\gamma_j)\int_{e_j}\zeta.
\end{aligned}
\]
Write the endpoints of $e_j$, in its orientation, as
$\delta_j^-o$ and $\delta_j^+o$, with
$\delta_j^\pm\in\Gamma$. Then
\[
 \int_{e_j}\zeta
 =\sigma_\zeta(\delta_j^+)
  -\sigma_\zeta(\delta_j^-).
\]
Consequently,
\[
 I_q([\omega],[\zeta])
 =-\sum_{j=1}^{2g}
 \chi_q(\gamma_j)^{-1}\sigma_\omega(\gamma_j)
 \bigl(\sigma_\zeta(\delta_j^+)
       -\sigma_\zeta(\delta_j^-)\bigr).
\]

For every fixed $\gamma\in\Gamma$, the cocycle relation
expresses $\sigma_\omega(\gamma)$ as a linear combination
of its values on the generators $a_i,b_i$ from \eqref{eq:tc-presentation}, with Laurent-polynomial
coefficients in $q$. The same holds for
$\sigma_\zeta(\gamma)$, with $q$ replaced by $q^{-1}$.
Thus the displayed formula is a finite bilinear
expression in the cocycle coordinates, with
Laurent-polynomial coefficients in $q$.
Using  the algebraic Betti trivializations from
\Cref{lem:tc-charts}, we get that $q\mapsto I_q$ is a regular
section of
$\Bil(\mathcal H_X^1,\iota^*\mathcal H_X^1)$.
\end{proof}

\subsection{The Hodge star is not rational in Betti coordinates}
\label{subsec-hdnonr}

Throughout this section, we fix $O\subseteq X$ a  Zariski-open subset containing the unitary locus $U$.  Our goal is to show the following.

\begin{proposition} \label{prop-irrat-Hodge}
The Hodge-star $q\mapsto \rot_{q}$ does \emph{not} extend into a regular section of the algebraic bundle $\mathrm{End}(\mathcal H^1_{X})_{|O}$. 
\end{proposition}

Note that combined with previous results, it shows singularity at infinity for finitely supported generating measures on torsion-free cocompact lattices of $\PSL_{2}(\R)$.

\begin{proof}[Proof of \Cref{thm:main}, torsion-free case]
If the Furstenberg measure of $\mu$ is not mutually singular with $\leb$, then we may apply \Cref{finite-notsing-rat-HS}, yielding a contradiction with \Cref{prop-irrat-Hodge}.
\end{proof}

\begin{remark}
\label{rem:rational-extension}
When establishing the torsion-free case of \Cref{thm:main}, the finite-support assumption on $\mu$ was only  used to establish the conclusion of \Cref{lem-finite-supp-ratext}, regarding the rationality  of the covariance form assignment $q\mapsto B_{\nu, \chi_{q}}$. All other parts of the argument only required finite exponential moment for $\mu$.  
Therefore, the contradiction applies more generally to any generating measure with
finite exponential moment satisfying the conclusion of \Cref{lem-finite-supp-ratext}. This observation will permit us in \Cref{sec:torsion} to generalize the surface group case to cocompact discrete subgroups $\Gamma< \mathrm{PSL}_{2}(\R)$ which may have  torsion.
\end{remark}

We now work toward establishing \Cref{prop-irrat-Hodge}. We will rely on the following standard fact. 

\begin{lemma}\label{lem:picard-zero}
Let $n\geq 1$. Every algebraic line bundle over a  Zariski-open subset of $(\C^{*})^n$ is trivial. 
\end{lemma}

\begin{proof} 
It follows from  \cite[Lemma, p.~794]{Popov23},
taking the smooth irreducible variety therein to be a point.
\end{proof}

We deduce that rationality of the Hodge-star would imply trivial first Chern class for the $(1,0)$ part of $\mathcal H^1_{U}$.
More precisely, for $q\in U$, set
\[
 H_q^{1,0}:=\{v\in H_q^1:\rot_qv=-iv\}.
\]
The latter corresponds to classes $[\omega]\in H_q^1$ whose harmonic representative $\omega^{\mathrm h}$ on $\D$  is $\C$-linear at every point in $\D$ (equivalently, $x\mapsto \omega^{\mathrm h}_{x}$ has holomorphic regularity). 
The subset $\mathcal H^{(1,0)}_{U} =\sqcup_{q\in U} H_q^{1,0}$ is a smooth subbundle of $\mathcal H^{1}_{U}$. Indeed, $P:=\frac12(\Id+i \rot)$ is a smooth idempotent endomorphism of $\mathcal H_U^1$, whose image is precisely $\mathcal H^{(1,0)}_{U}$, so \cite[Theorem~10.34]{Lee} justifies the  subbundle structure. 
The rank of the smooth idempotent $P$ is locally constant, hence constant on the connected space $U$. At a nontrivial character of order two, complex conjugation acts within the same cohomology space and exchanges the two eigenspaces of $\rot_q$. Since $\dim H_q^1=2g-2$, their dimensions are both $g-1$, and the same holds throughout $U$. Thus $\mathrm{rk}\mathcal H^{(1,0)}_U=g-1$.

\begin{lemma}\label{lem:hodge-c1-zero}
If \Cref{prop-irrat-Hodge} fails, then $\mathcal H^{(1,0)}_{U}$ has trivial first Chern class.
\end{lemma}

\begin{proof} Consider a regular section $\widetilde\rot$ of $\mathrm{End}(\mathcal H^1_{X})_{|O}$ whose restriction to $U$ is $q\mapsto \rot_{q}$. 
Since $U$ is Zariski dense in $O$ and $\rot_q^2=-\Id$ for all  $q\in U$, we have
$\widetilde\rot^{\,2}=-\Id$ on $O$. We may thus extend to $O$ the construction of $\mathcal H^{(1,0)}_{U}$ above, using 
$
\widetilde P:=\frac12(\Id+i\widetilde\rot),
$
and obtain an algebraic subbundle 
$\mathcal E :=\operatorname{Im}\widetilde{P}$
of $\mathcal H^1_{X|O}$, of rank $g-1$, extending $\mathcal H^{(1,0)}_{U}$. By \Cref{lem:picard-zero}, the
algebraic line bundle $\det\mathcal E:= \bigwedge^{g-1} \mathcal E$ is trivial over $O$ (whence over $U$). In particular,  as the first Chern class $c_{1}(\cdot )$ is insensitive to replacing a vector bundle by its determinant bundle, we obtain for the restriction to $U$:
\begin{equation}
\label{eq:c1zero-hodge}
c_1\bigl(\mathcal H^{(1,0)}_{U} \bigr)=0.
\end{equation}
\end{proof}

We now work toward showing that $\mathcal H^{(1,0)}_{U}$ cannot have trivial first Chern class. For that, we show it is smoothly isomorphic to the canonical Picard bundle over the (punctured) Jacobian of $S_{\Gamma}$. We briefly recall these objects and obtain nontrivial Chern class from the literature.

The \emph{Jacobian}  of $S_{\Gamma}$, denoted by $J$,  is by definition the set of isomorphism classes of holomorphic line bundles of degree $0$ over $S_{\Gamma}$. As explained in \cite[Sections~3.1--3.2]{GoldmanXia}, it can be identified with a complex torus $Z\backslash \C^g$ where $Z\subseteq \C^g$ is a lattice. Here $\C^g$ corresponds in fact to the $\C$-vector space $\mathscr{H}^{(0, 1)}(S_{\Gamma}, \C)=:\mathscr{H}^{(0,1)}$, consisting of the harmonic smooth one-forms on $S_{\Gamma}$ of type $(0, 1)$. For such a form $\alpha$, the associated holomorphic line bundle $L_{\alpha}$ on $S_{\Gamma}$ is obtained by considering the trivial smooth bundle $S_{\Gamma}\times \C$ over $S_{\Gamma}$ and declaring a local section $\sigma$ to be holomorphic if $\overline{\partial}_{S_{\Gamma}}\sigma(x)+  2i\pi\sigma(x)\alpha_{x}=0$  for all $x\in S_{\Gamma}$. Two forms $\alpha, \beta \in \mathscr{H}^{(0,1)}$ define isomorphic bundles 
if and only if $\alpha-\beta \in Z$ where $Z \subseteq \mathscr{H}^{(0, 1)}$ is the lattice of forms $\omega \in \mathscr{H}^{(0, 1)}$ such that $\int_{[o, \gamma o]}\overline{\omega} + \omega \in \Z$ for all $\gamma\in \Gamma$. 
This yields the complex torus identification $J\equiv Z\backslash \mathscr{H}^{(0, 1)}$. Moreover, the group structure on $J$ given by tensor product   corresponds to the additive structure on $Z\backslash \mathscr{H}^{(0, 1)}$, in other terms $L_{\alpha}\otimes L_{\beta}\simeq L_{\alpha +\beta}$. 

In our context, we have already encountered the Jacobian of $S_{\Gamma}$, but with a different description, via unitary characters. For $q\in (\C^{*})^{2g}$, recall the line bundle $L_{q}$ from \Cref{rem-linebdle}.


\begin{lemma}\label{Ujtimes}
The map $(S^1)^{2g}\rightarrow J, q\mapsto [L_{q}]$ is a  real Lie group isomorphism. 
\end{lemma}

In particular, we obtain a diffeomorphism
$$U\simeq J^\times $$
where $J^\times=J\smallsetminus \{\mathcal O_{S_{\Gamma}}\}$ with $\mathcal O_{S_{\Gamma}}$  the trivial bundle.
For future reference, we denote by $(\mathscr{H}^{(0, 1)})^\times =\mathscr{H}^{(0, 1)} \smallsetminus Z$ the preimage of $J^\times$ in $\mathscr{H}^{(0, 1)}$.

\begin{proof}[Proof of \Cref{Ujtimes}]
We start by giving an alternative description of $[L_{q}]$. For $q\in (S^{1})^{2g}$, let $\eta_{q}$ be any  $\Gamma$-invariant harmonic real one-form on $\D$ such that $\chi_{q}(\gamma)=\exp(2i \pi \int_{[o,\gamma o]} \eta_{q})$ for all $\gamma\in \Gamma$. Such a form exists by the de Rham and Hodge theorems. Set $\alpha_{q}:=\eta^{(0, 1)}_{q}$ its $(0, 1)$-part. Then $L_{q}$ is isomorphic with $L_{\alpha_{q}}$. Indeed, set $E_{\eta_{q}}:\D\rightarrow S^1, x\mapsto \exp(2i\pi \int_{[o,x]} \eta_{q})$. Note $E_{\eta_{q}}(\gamma x)=\chi_{q}(\gamma)E_{\eta_{q}}( x)$.  Given $f:\D\rightarrow \C$ smooth, set $\sigma_{f}=E_{\eta_{q}}^{-1}f$. Then $f$ is $\chi_{q}$-equivariant if and only if $\sigma_{f}$ is $\Gamma$-invariant, and  we have $\frac{\partial f}{\partial \overline{z}}=0$ if and only if $\bar\partial\sigma_f+2\pi i\,\alpha_q\sigma_f=0$.  It follows that $\D\times \C\rightarrow \D\times \C, (x, z)\mapsto (x, E^{-1}_{\eta_{q}}(x)z)$  passes to $\Gamma$-quotients to give an isomorphism of holomorphic bundles $L_{q}\simeq L_{\alpha_{q}}$.
We are thus reduced to considering  the map $q\mapsto [L_{\alpha_{q}}]$. 

That $q\mapsto [L_{\alpha_{q}}]$ is a homomorphism follows from choosing $\eta_{q q'}:=\eta_{q}+\eta_{q'}$. Smoothness is justified by choosing the assignment  $q\mapsto \eta_{q}$  to depend smoothly on $q$ near any given $q_{0}\in (S^1)^{2g}$. 

For injectivity, suppose $L_q$ is holomorphically trivial.
Then there
exists a nowhere vanishing holomorphic function
$f:\D\to\C$ satisfying
\[
    f(\gamma x)=\chi_{q}(\gamma)f(x).
\]
Since this character is unitary, $|f|$ is $\Gamma$-invariant
and therefore attains a global maximum, by compactness of
$S_\Gamma$. The maximum principle implies that $f$ is
constant. Since $f$ is nonzero, the character is trivial,
hence $q=1$.

Surjectivity follows from the above, because  $(S^1)^{2g}$ and $J$ are connected Lie groups of same dimension.

 \end{proof}

The bundle $\mathcal H^{(1,0)}_U$ over $U$ has an analogue
over $J^\times$, given by the canonical Picard bundle.
We recall its construction; see
\cite[Sections~7.1--7.2]{Gunning18} for further details.
Write $K_{S_\Gamma}$ for the canonical holomorphic line bundle
on $S_\Gamma$, whose fibre at $x$ is the space of $\C$-linear
maps from $T_xS_\Gamma$ to $\C$.
For $\alpha\in(\mathscr H^{(0,1)})^\times$, consider the
vector space
\[
\widetilde{\mathcal F}_\alpha
:=
H^0(S_\Gamma,K_{S_\Gamma}\otimes L_\alpha)
\]
of holomorphic one-forms on $S_\Gamma$ with values in $L_\alpha$. Using the fixed smooth trivialization of $L_\alpha$, we may see 
$\widetilde{\mathcal F}_\alpha$ as the space of smooth
$(1,0)$-forms $\omega$ on $S_\Gamma$ satisfying
$\bar\partial_{S_\Gamma}(\omega(v))
+2\pi i\,\alpha\,\omega(v)=0$
for every local holomorphic vector field $v$ on $S_\Gamma$. Since $\alpha\notin Z$, the degree-zero bundle $L_\alpha$
is nontrivial, so $H^0(S_\Gamma,L_\alpha^{-1})=0$.
The Riemann--Roch theorem therefore gives
\[
\dim \widetilde{\mathcal F}_\alpha =g-1.
\]
Set $\widetilde{\mathcal F}=\sqcup_{\alpha\in(\mathscr H^{(0,1)})^\times} \widetilde{\mathcal F}_\alpha$. We endow $\widetilde{\mathcal F}$ with a structure of holomorphic bundle over $(\mathscr H^{(0,1)})^\times$. For that,  a local (set-theoretic) section $\alpha\mapsto \omega^{(\alpha)}$ is declared
holomorphic if, for any smooth vector field $v$ on $S_{\Gamma}$, the map $(x, \alpha) \mapsto \omega^{(\alpha)}_{x}(v(x))$ is smooth
jointly in $(x,\alpha)$, and holomorphic in $\alpha$.
By \cite[Theorem~7.5]{Gunning18}, the above prescription
defines a holomorphic vector bundle $\widetilde{\mathcal F}$
of rank $g-1$ over $(\mathscr H^{(0,1)})^\times$.

For $\beta\in Z$, the bundle $L_\beta$ is holomorphically
trivial. Let $u_\beta$ be its unique nowhere vanishing
holomorphic section normalized by $u_\beta(x_0)=1$,
where $x_0=\Gamma o$.
Using the fixed smooth trivialization of $L_\beta$, we view
$u_\beta$ as a smooth function $S_\Gamma\to\C^*$.
Direct computation 
shows that  the multiplication map $\omega \mapsto u_\beta \omega$ gives an isomorphism
from $\widetilde{\mathcal F}_\alpha$ onto
$\widetilde{\mathcal F}_{\alpha+\beta}$.
Moreover, the normalization at $x_0$ implies
\[
u_0=1,
\qquad
u_{\beta+\beta'}=u_\beta u_{\beta'}
\quad\text{for all }\beta,\beta'\in Z.
\]
Thus $Z$ acts on $\widetilde{\mathcal F}$ by
\[
\beta\cdot(\alpha,\omega)
=
(\alpha+\beta,u_\beta \omega).
\]
This action is holomorphic.
Since $Z$ acts freely and properly discontinuously on the
base, the quotient $
\mathcal F:=Z\backslash\widetilde{\mathcal F}
$
is a holomorphic vector bundle of rank $g-1$ over
$
J^\times\simeq
Z\backslash(\mathscr H^{(0,1)})^\times.
$
We call $\mathcal F$ the \emph{canonical Picard bundle}
over $J^\times$.\footnote{Here ``canonical'' refers to
the choice of the canonical bundle $K_{S_\Gamma}$. The construction can indeed be adapted for other holomorphic line bundles over $S_{\Gamma}$, of large enough degree.}

\bigskip

\bigskip
We now identify  $\mathcal H^{(1,0)}_{U}$ with $\mathcal F$.

\begin{lemma}
The assignment $\mathcal H^{(1,0)}_{U} \rightarrow \widetilde{\mathcal F}$ given by $(q, [\omega])\mapsto (\alpha_{q}, E^{-1}_{\eta_{q}}\omega^{\mathrm h})$ descends, by post-composing with the projection $\widetilde{\mathcal F} \rightarrow \mathcal F$, into  an isomorphism of smooth bundles 
$$\mathcal H^{(1,0)}_{U} \simeq  \mathcal F.$$
\end{lemma}

\begin{proof} The assignment is well-defined. Indeed, $\omega^{\mathrm h}$ is a holomorphic $\chi_{q}$-equivariant one-form on $\D$, so up to quotienting by $\Gamma$, we may see it as an element of $H^0(S_{\Gamma}, K_{S_{\Gamma}}\otimes L_q)$, and moreover, we have an isomorphism $ L_q  \simeq L_{\alpha_{q}}$ coming from the map $(x, z)\mapsto (x,E^{-1}_{\eta_{q}}z)$ as seen in the proof of \Cref{Ujtimes}. The assignment is  clearly a set-theoretic bundle morphism $\mathcal H^{(1,0)}_{U} \rightarrow \widetilde{\mathcal F}$. Moreover it is fibrewise invertible, and the quotient map  $\Phi:\mathcal H^{(1,0)}_{U} \rightarrow  \mathcal F$ is a bijection on the bases $U\rightarrow J^\times$, by \Cref{Ujtimes}. Hence $\Phi$ must be a set-theoretic bundle isomorphism. In fact, $\Phi$ does not depend on the choice of the $(0, 1)$-form  $\alpha_{q}$ representing $\chi_{q}$, so the latter can be chosen to depend smoothly on $q$ near any given $q_{0}\in U$. With \Cref{lem:hodge}, this shows that  $\Phi$ is smooth. 
\end{proof}

\begin{lemma} \label{lem-nontrivialc1}
The vector bundle $\mathcal F$ over $J^\times$ has nontrivial first Chern class.
\end{lemma}

\begin{proof}
Allowing a simple pole at $x_0$ in the construction of
$\mathcal F$ gives a holomorphic vector bundle $\mathcal E$
over all of $J$, with fibres
\[
\mathcal E_{[L_{\alpha}]}
\simeq
H^0(S_\Gamma,K_{S_\Gamma}(x_0)\otimes L_{\alpha}).
\]
Here $K_{S_\Gamma}(x_0)$ denotes the bundle of meromorphic
one-forms with at most a simple pole at $x_0$.
By Riemann--Roch, $\mathcal E$ has rank $g$.
Since $\deg K_{S_\Gamma}(x_0)=2g-1$, the classical negativity
theorem for normalized Picard bundles implies that
$c_1(\mathcal E)\neq0$; see
\cite[Theorem~6.3.48 and Corollary~6.1.16]{LazarsfeldII}. We now pull back this information to $\mathcal F$.

For $[L]\in J^\times$, let $\operatorname{Res}_{x_0, [L]} : \mathcal E_{[L]} \rightarrow \C$ be the map sending a meromorphic form to its residue at $x_{0}$. It is well defined due to the normalization $u_\beta(x_0)=1$ in the quotient definition of $\mathcal E$, and it   
has kernel
$H^0(S_\Gamma,K_{S_\Gamma}\otimes L)$.
This kernel has dimension $g-1$, so the residue map is
surjective. This gives an exact
sequence of vector bundles
\[
0\longrightarrow\mathcal F
\longrightarrow\mathcal E|_{J^\times}
\xrightarrow{\operatorname{Res}_{x_0}}
J^\times\times\C
\longrightarrow0.
\]
By additivity of $c_{1}(\cdot)$ along short exact sequences, we get $c_1(\mathcal F)=c_1(\mathcal E|_{J^\times})$. But $\mathcal E|_{J^\times}=i^*\mathcal E$, where
$i:J^\times\hookrightarrow J$ is the inclusion, so by naturality, we obtain
$c_1(\mathcal F)=i^*c_1(\mathcal E)$.
Since $J$ has real dimension $2g\ge4$, excision shows that
$
i^*:H^2(J;\mathbb Z)\longrightarrow H^2(J^\times;\mathbb Z)
$
is an isomorphism. Hence $c_1(\mathcal F)\neq0$.
\end{proof}

We may now conclude that the Hodge-star endomorphism of $\mathcal H^1_{U}$ is not rational in Betti coordinates.

\begin{proof}[Proof of \Cref{prop-irrat-Hodge}]
A regular extension would imply
$c_1(\mathcal H_U^{(1,0)})=0$
by \Cref{lem:hodge-c1-zero}.
The identification with the canonical Picard bundle
and \Cref{lem-nontrivialc1} give a contradiction.

\end{proof}


\section{Allowing the group to have torsion}
\label{sec:torsion}

Let $\widetilde\Gamma<\PSL_2(\R)$ now be a cocompact discrete subgroup, possibly
with torsion, and let $\widetilde \mu$ be a finitely supported  generating measure on $\widetilde\Gamma$.
We show the associated Furstenberg measure $\nu$ on $\partial \D$ is singular with respect to $\leb$.

The case where $\widetilde \Gamma$ is torsion-free was handled in previous sections. We reduce to this case.  By Selberg's lemma
\cite{Nica}, $\widetilde \Gamma$ has a torsion-free finite-index subgroup $\Gamma \leq \widetilde\Gamma$. Let $\mu:={\widetilde \mu}_{\Gamma}$ be the measure on $\Gamma$ induced by $\widetilde \mu$. By \cref{lem:returns} and \Cref{rem-return}, $\mu$ is  a  generating measure on $\Gamma$ with finite exponential moment, and it has the same Furstenberg measure $\nu$ as $\widetilde \mu$. 

Now, by \Cref{rem:rational-extension}, showing $\nu \perp \leb$ reduces to proving the rationality of the winding covariance pairing $q \mapsto B_{\nu, \chi_{q}}$. More precisely, recalling the  bundle constructions of the previous sections, we need to show the following.

\begin{lemma}\label{lem-induced-finite}
The winding covariance  assignment $q\mapsto B_{\nu, \chi_{q}}$ defined on $\Utor$ extends to a regular section of the restricted bundle $\Bil(\mathcal H^1_{X}, \iota^*\mathcal H^1_{X})_{|O}$ over a Zariski-open subset $O\subseteq X$ containing $U$. 
\end{lemma}

\Cref{thm:main} is a formal consequence as follows. 

\begin{proof}[Proof of \cref{thm:main}]
Combine \Cref{rem:rational-extension} (which relies on all previous sections) and \Cref{lem-induced-finite}.
\end{proof}

We now engage in the proof of \Cref{lem-induced-finite}.

\begin{proof}[Proof of \Cref{lem-induced-finite}]
Let $X_n=Y_1\cdots Y_n$, where the $Y_i$ are independent
$\widetilde\Gamma$-valued variables with law $\widetilde\mu$,
and $X_0=e$. Let
$
 \tau=\min\{n\geq1:X_n\in\Gamma\},$
so that $X_\tau$ has law $\mu$.
Let $q\in\Utor$, and let $\sigma,\psi:\Gamma\to\C$ be
respectively a $\chi_q$-cocycle and a $\chi_{q^{-1}}$-cocycle.
We identify their classes with elements of $H_q^1$ and
$H_{q^{-1}}^1$ using \Cref{lem:tc-periods}, and denote their
$\mu$-centered representatives by $\sigma^0$ and $\psi^0$,
as in \Cref{lem:centered}. By \eqref{eq:covariance-centered},
\begin{equation}\label{eq-bnuchiqo}
 B_{\nu,\chi_q}([\sigma],[\psi])
 =\frac1{L_\mu}\E\bigl[
       \sigma^0(X_\tau)\psi^0(X_\tau)\bigr].
\end{equation}
We must show that this expression depends regularly on $q$ in
the Betti coordinates on a Zariski-open neighborhood of $U$.

To exploit the finite support of $\widetilde\mu$, we will express
the expectation above as a finite sum involving $\widetilde\mu$,
thereby removing the stopping time.
As in the proof of \Cref{lem:stopped-covariance}, the idea is
to realize the two factors as stopped martingales and use
orthogonality of their increments.
Here, however, the intermediate positions $X_n$ need not belong
to $\Gamma$, so $\sigma$ and $\psi$ cannot be evaluated directly
along the walk.
We therefore first induce them to vector-valued cocycles on
$\widetilde\Gamma$, and then center these induced cocycles
with respect to $\widetilde\mu$.

Introduce the finite coset space
$F=\Gamma\backslash\widetilde\Gamma$, with its right
$\widetilde\Gamma$-action, and choose representatives
$(t_k)_{k\in F}$ in $\widetilde\Gamma$ with $t_{[e]}=e$.
For $k\in F$ and $g\in\widetilde\Gamma$, put
\[
 h(k,g)=t_kgt_{kg}^{-1}\in\Gamma.
\]
and note the relation
$
 h(k,gg')=h(k,g)h(kg,g')$.
For any $q\in X$, the character $\chi_q$ induces a representation
$\rho_q$ of $\widetilde\Gamma$ on $\C^F$, defined by
\[
 (\rho_q(g)v)(k)
 =\chi_q(h(k,g))v(kg)
 \qquad(g\in\widetilde\Gamma,\ v\in\C^F,\ k\in F).
\]
A cocycle for this representation is a map
$C:\widetilde\Gamma\to\C^F$ satisfying
\[
 C(gg')=C(g)+\rho_q(g)C(g'),
\]
and a coboundary is a map of the form
$g\mapsto\rho_q(g)v-v$ for some $v\in\C^F$.
Given a $\chi_q$-cocycle $\sigma:\Gamma\to\C$, the associated induced map $ C_\sigma:\widetilde \Gamma\rightarrow \C^F$ given by
\[
 C_\sigma(g)(k):=\sigma(h(k,g))
\]
is such a cocycle. It recovers $\sigma$ through the
base-coset coordinate:
\[
 C_\sigma(\gamma)([e])=\sigma(\gamma)
 \qquad(\gamma\in\Gamma).
\]
In particular, $\sigma(X_\tau)=C_\sigma(X_\tau)([e])$.
We define $C_\psi$ similarly, using $\rho_{q^{-1}}$.

We next center these induced cocycles with respect to
$\widetilde\mu$.
All sums over $g$ below range over $\supp\widetilde\mu$.
Let $ P_q$ be the averaging operator on $\C^F$ given by
\[
 P_q=\sum_g\widetilde\mu(g)\rho_q(g).
\]
We claim that, provided $q\in U$, the operator $\Id-P_q$ on $\C^F$ is invertible.
Indeed, $\rho_q$ is then unitary for the standard Hermitian norm
on $\C^F$. If $P_qv=v$, then by strict convexity of balls in a Hilbert space, 
 $v$ must be fixed by $\supp\widetilde\mu$, hence by
$\widetilde\Gamma$.
Evaluating invariance under $t_k$ and under $\gamma\in\Gamma$
at $[e]$ gives respectively
\[
 v(k)=v([e]),\qquad
 \chi_q(\gamma)v([e])=v([e]).
\]
Since $\chi_q$ is nontrivial, this forces $v=0$, proving the claim.

Consequently, setting $
 b_\sigma=\sum_g\widetilde\mu(g)C_\sigma(g)$ and  $v_\sigma=(\Id-P_q)^{-1}b_\sigma$,
we obtain the unique $\widetilde\mu$-centered cocycle cohomologous
to $C_\sigma$:
\[
 C_\sigma^0(g)
 :=C_\sigma(g)+\rho_q(g)v_\sigma-v_\sigma,
 \quad \text{ which satisfies }\quad
 \sum_g\widetilde\mu(g)C_\sigma^0(g)=0.
\]
Define $C_\psi^0$ in the same way using $P_{q^{-1}}$.

The sequence $(C_\sigma^0(X_n))_{n\geq0}$ is a $\C^F$-valued martingale
with bounded increments. Indeed, the cocycle identity shows that
\[
 C_\sigma^0(X_{n+1})-C_\sigma^0(X_n)
 =\rho_q(X_n)C_\sigma^0(Y_{n+1}),
\]
and $\rho_q$ is unitary, $\widetilde\mu$ has finite support, $C_\sigma^0$ is $\widetilde \mu$-centered. 
As $\E\tau<\infty$ by \Cref{lem:returns,rem-return},
we may sum these increments up to $\tau$ and take expectations:
\[
 \E[C_\sigma^0(X_\tau)]
 =\sum_{n\geq0}\E\bigl[
   \mathbf1_{\{n<\tau\}}\rho_q(X_n)C_\sigma^0(Y_{n+1})
 \bigr]=0.
\]
On the other hand, for $\gamma\in\Gamma$,
\[
 C_\sigma^0(\gamma)([e])
 =\sigma(\gamma)+(\chi_q(\gamma)-1)v_\sigma([e]).
\]
Thus $\gamma\mapsto C_\sigma^0(\gamma)([e])$ is cohomologous
to $\sigma$ and, since $X_\tau$ has law $\mu$, has zero
$\mu$-average.
Uniqueness in \Cref{lem:centered} shows that it equals $\sigma^0$.
The same argument applies to $\psi$. In particular,
\[
 \sigma^0(X_\tau)=C_\sigma^0(X_\tau)([e]),\qquad
 \psi^0(X_\tau)=C_\psi^0(X_\tau)([e]).
\]

We can now compute the required expectation in \eqref{eq-bnuchiqo} using the scalar
martingales
\[
 M_n^\sigma=C_\sigma^0(X_n)([e]),\qquad
 M_n^\psi=C_\psi^0(X_n)([e]).
\]
They start at zero and have terminal values
$M_\tau^\sigma=\sigma^0(X_\tau)$ and
$M_\tau^\psi=\psi^0(X_\tau)$.
Writing $K_n=\Gamma X_{n}\in F$ for the projection of $X_{n}$ to $F$,  their increments satisfy
\[
 \begin{aligned}
 M_{n+1}^\sigma-M_n^\sigma
 &=\chi_q(X_nt^{-1}_{K_{n}} )C_\sigma^0(Y_{n+1})(K_n),\\
 M_{n+1}^\psi-M_n^\psi
 &=\chi_{q^{-1}}(X_n t^{-1}_{K_{n}})C_\psi^0(Y_{n+1})(K_n).
 \end{aligned}
\]
The character factors cancel in the product of these increments.
Orthogonality of martingale increments and the finite exponential moment of $\tau$ give
\[
 \E\bigl[\sigma^0(X_\tau)\psi^0(X_\tau)\bigr]
 =\E\sum_{n<\tau}\sum_g\widetilde\mu(g)
     C_\sigma^0(g)(K_n)C_\psi^0(g)(K_n).
\]
The coset walk $(K_n)$ is irreducible, with uniform stationary distribution on $F$, so the occupation
formula for a return cycle \cite[Theorem~1.7.6]{Norris} gives
\[
 \E\sum_{n<\tau}\mathbf1_{\{K_n=k\}}=1
 \qquad(k\in F).
\]
Therefore, for every $q\in \Utor$, we   get from \eqref{eq-bnuchiqo} that
\[
 B_{\nu,\chi_q}([\sigma],[\psi])
 =\frac1{L_\mu}
   \sum_{k\in F}\sum_g\widetilde\mu(g)
        C_\sigma^0(g)(k)C_\psi^0(g)(k).
\]

This finite expression gives the desired regular extension.
Indeed, the entries of $P_q$ are Laurent polynomials in $q$.
Thus
\[
 O=\{q\in X:
       \det(\Id-P_q)\det(\Id-P_{q^{-1}})\ne0\}
\]
is a Zariski-open subset containing $U$, stable under inversion.
For every $q\in O$, we may employ the same finite formulas to define $C_\sigma^0$
and $C_\psi^0$ and those only depend\footnote{Indeed, induction sends coboundaries to coboundaries, and centering
annihilates them. } on the classes $[\sigma]$, $[\psi]$. The displayed finite expression above
therefore defines a bilinear pairing
$\widetilde B_q([\sigma],[\psi])$ on
$H_q^1\times H_{q^{-1}}^1$ for every $q\in O$, which extends $ B_{\nu,\chi_q}$ on $\Utor$.
Finally, the cocycle identity and \Cref{lem:tc-charts} show that $\widetilde B$ is a regular section of
$\Bil(\mathcal H_X^1,\iota^*\mathcal H_X^1)_{|O}$, thus finishing the proof.
\end{proof}


\bigskip
{\footnotesize
\noindent
\textsc{CNRS -- LAGA, Universit\'e Sorbonne Paris Nord,
99 avenue J.-B. Cl\'ement, 93430 Villetaneuse}
\par\nobreak
\noindent
\textit{Email address: }
\href{mailto:benard@math.univ-paris13.fr}
     {\texttt{benard@math.univ-paris13.fr}}
\par
}


\small

\begin{thebibliography}{99}

\bibitem{BaranyPollicottSimon}
B.~B{\'a}r{\'a}ny, M.~Pollicott and K.~Simon,
\emph{Stationary measures for projective transformations: the Blackwell and Furstenberg measures},
J. Stat. Phys. \textbf{148} (2012), no.~3, 393--421.

\bibitem{Beardon0}
A.~F. Beardon,
\emph{Inequalities for certain Fuchsian groups},
Acta Math. \textbf{127} (1971), 221--258.

\bibitem{Benard}
T.~B\'enard,
\emph{Winding of geodesic rays chosen by a harmonic measure},
Math. Ann. \textbf{390} (2024), 1419--1465.
\href{https://doi.org/10.1007/s00208-023-02754-z}{doi:10.1007/s00208-023-02754-z}.
\href{https://arxiv.org/abs/2211.15232v3}{arXiv:2211.15232v3}.

\bibitem{BQ}
Y.~Benoist and J.-F.~Quint,
\emph{On the regularity of stationary measures},
Israel J. Math. \textbf{226} (2018), no.~1, 1--14.
\href{https://www.imo.universite-paris-saclay.fr/~yves.benoist/prepubli/15regularity.pdf}
{Authors' preprint}.

\bibitem{BlachereHaissinskyMathieu}
S.~Blach\`ere, P.~Ha{\"i}ssinsky and P.~Mathieu,
\emph{Harmonic measures versus quasiconformal measures for hyperbolic groups},
Ann. Sci. \'Ec. Norm. Sup\'er. (4) \textbf{44} (2011), no.~4, 683--721.

\bibitem{Bogachev26}
N.~Bogachev,
\emph{Singularity of harmonic measures for hyperbolic lattices},
preprint (2026),
\href{https://arxiv.org/abs/2609.40357}{arXiv:2609.40357}.

\bibitem{BKT}
N.~Bogachev, P.~Kosenko, and G.~Tiozzo,
\emph{Random walks on cocompact Fuchsian and Kleinian groups},
preprint (2025),
\href{https://arxiv.org/abs/2512.09900v1}{arXiv:2512.09900v1}.


\bibitem{Bourgain}
J.~Bourgain,
\emph{Finitely supported measures on $SL_2(\R)$ which are absolutely
continuous at infinity},
in \emph{Geometric Aspects of Functional Analysis},
Lecture Notes in Mathematics~2050, Springer, 2012, 133--141.
\href{https://doi.org/10.1007/978-3-642-29849-3_7}
{doi:10.1007/978-3-642-29849-3\_7}.

\bibitem{CLP}
M.~Carrasco, P.~Lessa, and E.~Paquette,
\emph{On the speed of distance-stationary sequences},
ALEA Lat. Am. J. Probab. Math. Stat. \textbf{18} (2021), 829--854.
\href{https://doi.org/10.30757/ALEA.v18-30}
{doi:10.30757/ALEA.v18-30}.

\bibitem{CFFT}
K.~Chawla, B.~Forghani, J.~Frisch, and G.~Tiozzo,
\emph{The Poisson boundary of hyperbolic groups without moment conditions},
Ann. Probab. \textbf{53} (2025), no.~5, 1897--1918.
\href{https://doi.org/10.1214/24-AOP1753}{doi:10.1214/24-AOP1753}.

\bibitem{DeroinKleptsynNavas}
B.~Deroin, V.~Kleptsyn and A.~Navas,
\emph{On the question of ergodicity for minimal group actions on the circle},
Mosc. Math. J. \textbf{9} (2009), no.~2, 263--303.

\bibitem{Furstenberg63}
H.~Furstenberg,
\emph{Noncommuting random products},
Trans. Amer. Math. Soc. \textbf{108} (1963), 377--428.

\bibitem{Furstenberg}
H.~Furstenberg,
\emph{Random walks and discrete subgroups of Lie groups},
in \emph{Advances in Probability and Related Topics}, Vol.~1,
Marcel Dekker, New York, 1971, 1--63.

\bibitem{GMT}
V.~Gadre, J.~Maher, and G.~Tiozzo,
\emph{Word length statistics and Lyapunov exponents for Fuchsian groups
with cusps},
New York J. Math. \textbf{21} (2015), 511--531.
\href{https://nyjm.albany.edu/j/2015/21-23.html}{Journal page}.

\bibitem{Goldman}
W.~M.~Goldman,
\emph{The symplectic nature of fundamental groups of surfaces},
Adv. Math. \textbf{54} (1984), no.~2, 200--225.
\href{https://doi.org/10.1016/0001-8708(84)90040-9}{doi:10.1016/0001-8708(84)90040-9}.

\bibitem{GLJ90}
Y.~Guivarc'h and Y.~Le Jan,
\emph{Sur l'enroulement du flot g\'eod\'esique},
C. R. Acad. Sci. Paris S\'er. I Math. \textbf{311} (1990), no.~10, 645--648.

\bibitem{GLJ93}
Y.~Guivarc'h and Y.~Le Jan,
\emph{Asymptotic winding of the geodesic flow on modular surfaces and
continuous fractions},
Ann. Sci. \'{E}cole Norm. Sup. (4) \textbf{26} (1993), no.~1, 23--50.
\href{https://www.numdam.org/item/ASENS_1993_4_26_1_23_0/}{Numdam}.

\bibitem{GLJ96}
Y.~Guivarc'h and Y.~Le Jan,
\emph{Note rectificative: ``Asymptotic winding of the geodesic flow on
modular surfaces and continuous fractions''},
Ann. Sci. \'{E}cole Norm. Sup. (4) \textbf{29} (1996), no.~6, 811--814.
\href{https://doi.org/10.24033/asens.1755}{doi:10.24033/asens.1755}.

\bibitem{Gunning18}
R.~C.~Gunning,
\emph{Some Topics in the Function Theory of Compact
Riemann Surfaces},
preliminary version, January~24, 2018.
Available at
\url{https://web.math.princeton.edu/~gunning/book.pdf}.


\bibitem{GoldmanXia}
W.~M. Goldman and E.~Z. Xia,
\emph{Rank one Higgs bundles and representations of fundamental
groups of Riemann surfaces},
Mem. Amer. Math. Soc. \textbf{193} (2008), no.~904.

\bibitem{KLP}
V.~A.~Kaimanovich and V.~Le Prince,
\emph{Matrix random products with singular harmonic measure},
Geom. Dedicata \textbf{150} (2011), 257--279.
\href{https://arxiv.org/abs/0807.1015}{arXiv:0807.1015}.

\bibitem{KZ}
D.~M.~Kim and A.~Zimmer,
\emph{A note on the singularity conjecture for infinite covolume discrete
subgroups}, preprint (2025),
\href{https://arxiv.org/abs/2508.05756v1}{arXiv:2508.05756v1}.

\bibitem{Koberda}
T.~Koberda,
\emph{Teichm\"uller polynomials, Alexander polynomials and finite covers of surfaces},
arXiv:1110.3746, 2011.

\bibitem{KodairaSpencer1960}
K.~Kodaira and D.~C.~Spencer,
\emph{On deformations of complex analytic structures, III.
Stability theorems for complex structures},
Ann. of Math. (2) \textbf{71} (1960), no.~1, 43--76.

\bibitem{Kosenko}
P.~Kosenko,
\emph{Fundamental inequality for hyperbolic Coxeter and Fuchsian groups equipped with geometric distances},
Int. Math. Res. Not. IMRN \textbf{2021}, no.~6, 4709--4728.

\bibitem{KosenkoFirstPassage}
P.~Kosenko,
\emph{Asymptotics of the first-passage function on free and Fuchsian groups},
Ann. Probab. \textbf{54} (2026), no.~2, 597--609.
\href{https://doi.org/10.1214/25-AOP1773}
{doi:10.1214/25-AOP1773}.

\bibitem{KosenkoComplex}
P.~Kosenko,
\emph{On a complex-analytic approach to stationary measures on $S^1$
with respect to the action of $PSU(1,1)$},
J. Funct. Anal. \textbf{291} (2026), no.~1, Paper No.~111476.
\href{https://doi.org/10.1016/j.jfa.2026.111476}
{doi:10.1016/j.jfa.2026.111476}.


\bibitem{KT}
P.~Kosenko and G.~Tiozzo,
\emph{The fundamental inequality for cocompact Fuchsian groups},
Forum Math. Sigma \textbf{10} (2022), Paper No.~e102.
\href{https://doi.org/10.1017/fms.2022.94}{doi:10.1017/fms.2022.94}.

\bibitem{KT26}
P.~Kosenko and G.~Tiozzo,
\emph{Singularity of harmonic measure for finitely supported
random walks}, preprint (2026),
\href{https://arxiv.org/abs/2609.19115}{arXiv:2609.19115}.

\bibitem{LazarsfeldII}
R.~Lazarsfeld,
\emph{Positivity in Algebraic Geometry II:
Positivity for Vector Bundles, and Multiplier Ideals},
Ergebnisse der Mathematik und ihrer Grenzgebiete,
3.~Folge, vol.~49,
Springer-Verlag, Berlin, 2004.

\bibitem{LVT}
H.~Lee, G.~Tiozzo, and W.~Van LimBeek,
\emph{Singularity of Furstenberg measure for infinite covolume discrete
subgroups in higher rank}, preprint (2025),
\href{https://arxiv.org/abs/2508.06329}{arXiv:2508.06329}.

\bibitem{Lee}
J.~M.~Lee,
\emph{Introduction to Smooth Manifolds}, 2nd ed.,
Graduate Texts in Mathematics~218, Springer, 2013.
\href{https://doi.org/10.1007/978-1-4419-9982-5}{doi:10.1007/978-1-4419-9982-5}.

\bibitem{Lequen}
F.~Lequen,
\emph{Absolutely continuous Furstenberg measures for finitely-supported random walks},
Geom. Dedicata \textbf{219} (2025), no.~2, Paper No.~36.

\bibitem{LNP}
J.~Li, F.~Naud, and W.~Pan, with an appendix by J.~Li,
\emph{Kleinian Schottky groups, Patterson--Sullivan measures, and Fourier decay},
Duke Math. J. \textbf{170} (2021), no.~4, 775--825.
\href{https://doi.org/10.1215/00127094-2020-0058}{doi:10.1215/00127094-2020-0058}.
\href{https://arxiv.org/abs/1902.01103}{arXiv:1902.01103}.

\bibitem{MT}
J.~Maher and G.~Tiozzo,
\emph{Random walks on weakly hyperbolic groups},
J. Reine Angew. Math. \textbf{742} (2018), 187--239.
\href{https://arxiv.org/abs/1410.4173}{arXiv:1410.4173}.



\bibitem{Nica}
B.~Nica,
\emph{Linear groups---Malcev's theorem and Selberg's lemma},
preprint (2013),
\href{https://arxiv.org/abs/1306.2385}{arXiv:1306.2385}.

\bibitem{Norris}
J.~R. Norris,
\emph{Markov Chains},
Cambridge Series in Statistical and Probabilistic Mathematics,
vol.~2, Cambridge University Press, Cambridge, 1997.

\bibitem{Popov23}
V.~L.~Popov,
\emph{Picard group of a connected affine algebraic group},
Russian Math. Surveys \textbf{78} (2023), no.~4, 794--796.
\href{https://doi.org/10.4213/rm10107e}{doi:10.4213/rm10107e}.


\bibitem{Quint}
J.-F.~Quint,
\emph{An overview of Patterson--Sullivan theory},
lecture notes, 2006.
\href{https://www.math.u-bordeaux.fr/~jquint/publications/courszurich.pdf}
{Author's online notes}.

\bibitem{RT}
A.~Randecker and G.~Tiozzo,
\emph{Cusp excursion in hyperbolic manifolds and singularity of harmonic measure},
J. Mod. Dyn. \textbf{17} (2021), 183--211.
\href{https://doi.org/10.3934/jmd.2021006}{doi:10.3934/jmd.2021006}.

\bibitem{Ratner}
M.~Ratner,
\emph{The central limit theorem for geodesic flows on $n$-dimensional
manifolds of negative curvature},
Israel J. Math. \textbf{16} (1973), 181--197.
\href{https://doi.org/10.1007/BF02757869}{doi:10.1007/BF02757869}.

\bibitem{Serre1977}
J.-P.~Serre,
\emph{Linear Representations of Finite Groups},
Graduate Texts in Mathematics~42,
Springer-Verlag, 1977.

\bibitem{Shafarevich}
I.~R.~Shafarevich,
\emph{Basic Algebraic Geometry 2:
Schemes and Complex Manifolds},
3rd ed., Springer, 2013.

\bibitem{Wells2008}
R.~O.~Wells, Jr.,
\emph{Differential Analysis on Complex Manifolds},
3rd ed., Graduate Texts in Mathematics~65,
Springer, 2008.

\end{thebibliography}
\end{document}